\documentclass[a4paper,12pt]{amsart} 
\usepackage{tikz-cd}
\usepackage[utf8]{inputenc}
\usepackage[T1]{fontenc}
\usepackage[english]{babel}
\usepackage{lmodern}
\usepackage{amsmath,amssymb,amscd,amsfonts,latexsym}
\usepackage{hyperref}
\usepackage{amsrefs}
\usepackage{mathrsfs}
\usepackage{graphicx}
\usepackage{caption}
\usepackage{graphicx}
\usepackage{subcaption}
\usepackage[new]{old-arrows}
\usepackage{wrapfig}
\usepackage{float}
\restylefloat{table}
\usepackage{diagbox}
\usepackage{tikz}
\usepackage{enumitem}
\usepackage{csquotes}
\usepackage{verbatim}
\usepackage{breqn}
\newtheorem{theorem}{Theorem}[section]
\newtheorem{proposition}{Proposition}[section]
\newtheorem{lemma}[theorem]{Lemma}
\newtheorem{corollary}[theorem]{Corollary}
\newtheorem*{theorem*}{Theorem}
\theoremstyle{definition}
\newtheorem{definition}{Definition}[section]
\newtheorem{example}[theorem]{Example}

\newtheorem{remark}[theorem]{Remark}

\numberwithin{equation}{section}

\date{}

\def \C {{\mathbb C}}

\def \C {{\mathbb {C}}}

\usepackage{amsmath}
\usepackage{amssymb}
\usepackage{amsfonts}
\usepackage{amsthm}
\usepackage{mathtools}

\newcommand{\psiinf}{\psi_{\infty}}
\newcommand{\norm}[1]{\left\lVert#1\right\rVert}

\begin{document}

\title[Fixed points and Holomorphic Structures]{Fixed points and Holomorphic Structures on Line Bundles over the Quantum Projective Line}

\author[I. Biswas]{Indranil Biswas}
\address{Department of Mathematics, Shiv Nadar University, Dadri 201314, Uttar Pradesh, India.}
\email{indranil.biswas@snu.edu.in, inrdranil29@gmail.com}

\author[S. Guin]{Satyajit Guin}
\address{Department of Mathematics and Statistics, IIT Kanpur, Kanpur 208016, Uttar Pradesh, India.}
\email{sguin@iitk.ac.in}

\author[P. Kumar]{Pradip Kumar}
\address{Department of Mathematics, Shiv Nadar University, Dadri 201314, Uttar Pradesh, India.}
\email{pradip.kumar@snu.edu.in}



\begin{abstract}

It has recently been observed that, in contrast to the classical case, holomorphic structures on line bundles over the quantum projective line are not uniquely
determined by degree. We formulate a fixed--point--theoretic framework for the analysis of flat $\overline\partial$-connections that define holomorphic structures on line bundles over the quantum projective line. Within this framework, we relate the existence of invertible solutions to the gauge equation associated with holomorphic structures precisely to the existence of fixed points, lying in the open unit ball, of certain nonlinear maps acting on an appropriate Banach space.

\end{abstract}
\maketitle


\section{Introduction}

Noncommutative geometry developed by Connes \cite{Connes} is a far-reaching generalization of classical differential geometry, which enables mathematical analysis of spaces beyond the scope of classical methods. Within this framework, the notion of a {\em spectral triple} plays a fundamental role, extending the concept of 
a Riemannian manifold to the noncommutative setting. Over the past three decades, substantial progress has been made in understanding the geometry of noncommutative spaces, leading to a wealth of applications, particularly in mathematical physics. By contrast, the development of noncommutative complex geometry remains comparatively 
nascent. To the best of our knowledge, this line of inquiry began with the work of Connes and Cuntz \cite{Connes-Cuntz}, which outlined a possible approach to the idea of a complex structure in noncommutative geometry, based on the notion of positive Hochschild cocycle on an involutive algebra. Subsequently, Fr\"ohlich, Grandjean, and Recknagel \cite{MR1695097} introduced an alternative perspective inspired by Witten's supersymmetric approach to Morse inequalities \cite{MR683171}. The paradigmatic 
example of the noncommutative torus has been studied by Polishchuk and Schwarz \cite{Polishchuk-Schwarz}, and later was systematically developed by Polishchuk in detail \cite{Polishchuk}, \cite{Polishchuk2006AnaloguesExponentialMap} (see also \cite{satyajit} for noncommutative $2n$-torus). Beyond the torus, numerous authors have  investigated a variety of examples of `noncommutative complex manifolds' arising from quantum homogeneous spaces; see, for instance, \cite{Beggs-Smith}, 
\cite{DAndreya-Dabrowski}, \cite{DAndreaLandi2013GeometryQuantumProjectiveSpaces}, \cite{KhalkhaliLandiVanSuijlekom2011}, \cite{KhalkhaliMoatadelro2011NCComplexGeometryQPS}, \cite{Buachalla}
and references therein. Research in this direction remains active and continues to develop rapidly.

In this article, we focus on the quantum projective line $\mathbb{C}\mathrm{P}_q^1$ and on \emph{holomorphic} line bundles over it, which in many aspects parallel the situation of the classical Riemann sphere. From topological point of view, the coordinate algebra $\mathcal{A}(\mathbb{C}\mathrm{P}_q^1)$ of the quantum projective line coincides with that of the standard Podle\'s sphere $S_q^2$ \cite{Podles1987QuantumSpheres}. The noncommutative differential geometry of the standard Podle\'s sphere has been extensively studied from a variety of perspectives; among these the works (\cite{AguilarKaad2018PodlesSpectralMetric}, \cite{Dabrowski-Sitarz}, \cite{Neshveyev}, \cite{Schmudgen-Wagner}) are particularly relevant for the present work. Khalkhali, Landi, and van Suijlekom initiated a systematic study of complex geometry on the standard Podleś sphere in \cite{KhalkhaliLandiVanSuijlekom2011}. Recall that in the quantum setting a vector bundle is understood as a finitely generated projective (left) module over the algebra, and a holomorphic structure on it is induced by a flat $\overline\partial$-connection. Starting from a natural algebraic complex structure already implicit in earlier work of Majid \cite{Majid}, they lifted this structure to the analytic level of holomorphic functions and holomorphic sections. In particular, they equipped line bundles over the quantum projective line with canonical holomorphic structures and showed that many features of the classical complex geometry of the complex projective line $\mathbb{C}\mathrm{P}^1$ persist under $q$-deformation.

Very recently, a surprising phenomenon was discovered in \cite{GravemanLaRueMacArthurPesinWei2025NonStandardHolomorphic}. There, the authors show that on each holomorphic line bundle $\mathcal{L}_n$ over $\mathcal{A}(\mathbb{C}\mathrm{P}_q^1)$, there exist not just one but at least countably many distinct gauge-equivalence classes of flat $\overline\partial$-connections. In other words, holomorphic structures on line bundles over $\mathbb{C}\mathrm{P}_q^1$ are not uniquely determined (up to gauge equivalence) by degree. This phenomenon is genuinely quantum and has no classical analogue. To better appreciate this distinction, it is instructive to recall the classical situation of $\mathbb{C}\mathrm{P}^1$. Two holomorphic structures on a line bundle $\mathcal{O}(n)$ over $\mathbb{C}\mathrm{P}^1$ are gauge equivalent if the gauge equation
$\overline\partial g\,=\,g\overline\partial f-\overline\partial h\,.\,g$ admits an invertible solution $g$, where $f,\,h\,\in\, C^\infty(\mathbb{C}\mathrm{P}^1)$ are associated with the given $\bar\partial$-connections. This equation always admits an invertible solution, namely $g=\exp(f-h)$. Consequently, the holomorphic structure on $\mathcal{O}(n)$ is uniquely determined. This fundamental fact underlies the classical result that the Picard group of $\mathbb{C}\mathrm{P}^1$ is $\mathbb{Z}$. By contrast, in the quantum case the main result of \cite{GravemanLaRueMacArthurPesinWei2025NonStandardHolomorphic} shows that there exist multiple holomorphic structures on each line bundle $\mathcal{L}_n$ over $\mathcal{A}(\mathbb{C}\mathrm{P}_q^1)$. We emphasize a crucial point: the existence of solutions to the gauge equation is not the only difficulty; rather, the more subtle issue concerns the \emph{invertibility} of such solutions. In the quantum setting, this becomes more challenging because the solutions are no longer functions, as in the classical case, but elements of a noncommutative algebra. To indicate the level of complexity, it turns out that holomorphic structures corresponding to $f$ and $\alpha f$ need not always be gauge equivalent for a given $f\in\mathcal{A}(\mathbb{C}\mathrm{P}^1_q)$ and $\alpha\in\mathbb{C}^\times$ (see Lemma \ref{complicacy} and Example \ref{example}); such a situation never occurs in the classical case. Consequently, the complex geometry of $\mathbb{C}\mathrm{P}_q^1$ is significantly more intricate than that of its classical counterpart $\mathbb{C}\mathrm{P}^1$.

The concrete method in \cite{GravemanLaRueMacArthurPesinWei2025NonStandardHolomorphic} relies on restricting attention to a particular commutative unital subalgebra of the $C^*$-algebra $C(S_q^2)$ of the standard Podle\'s sphere. Let $B_0$ denote the standard self-adjoint generator of  $C(S_q^2)$ (see Section~2.3 for the notation), and consider $\mathcal{B}:=C^*\{1,B_0\}\subseteq C(S_q^2)$. Then $\mathcal{B}$ is isomorphic to the algebra of continuous functions on $\sigma(B_0)=\{0\}\cup\{q^{2n}:n\in\mathbb{N}\}$ (see \cite{Podles1987QuantumSpheres}, or \cite[Proposition 4.1]{AguilarKaad2018PodlesSpectralMetric}),
rendering this commutative $C^*$-subalgebra of $C(S_q^2)$ particularly tractable from an analytic perspective. Denote by $\mathrm{Pol}(B_0)$ the unital subalgebra of $\mathcal{B}\cong C(\sigma(B_0))$ consisting of polynomials in $B_0$. Then the restriction of the holomorphic differential $\overline\partial$ on $\mathcal{A}(\mathbb{C}\mathrm{P}_q^1)$ to $\mathrm{Pol}(B_0)$ is well-defined, since $B_0\in\mathcal{A}(\mathbb{C}\mathrm{P}_q^1)$. Then $\overline\partial$ acts, up to the factor $\overline\partial B_0$, analogously to the Jackson derivative (commonly referred to as the $q$-derivative), and one can effectively analyze the gauge equation by restricting attention to $\mathrm{Pol}(B_0)$. This is the theme of \cite{GravemanLaRueMacArthurPesinWei2025NonStandardHolomorphic}. However, these advantages are lost as soon as one moves beyond $\mathrm{Pol}(B_0)$. Outside of this commutative subalgebra, the methods of \cite{GravemanLaRueMacArthurPesinWei2025NonStandardHolomorphic} no longer apply, since the continuous functional calculus cannot be used. This limitation makes it clear that fundamentally new abstract methods are required in order to investigate the gauge equation over $\mathcal{A}(\mathbb{C}\mathrm{P}_q^1)$ and to probe its deeper structural features.

We formulate a fixed--point--theoretic framework to further analyze holomorphic structures on line bundles over the quantum projective line. Our approach begins with the identification of an appropriate Banach algebra $\mathscr{D}$, which turns out to be a linear subspace of the $C^*$-algebra $C(\mathbb{C}\mathrm{P}_q^1)$, but equipped with a certain ``graph norm'' different from the $C^*$-norm induced from $C(\mathbb{C}\mathrm{P}_q^1)$ (see Section $3$). We systematically formulate the gauge and the mixed gauge equation (see Def. \ref{main definition}), and then proceed to the fixed--point formulation of gauge equivalence of holomorphic structures (Section $4$). By applying the Banach fixed--point theorem, we establish a sufficient criterion for the gauge equivalence of holomorphic structures on $\mathcal{L}_n$. By example (see Example \ref{example}), we show how this criterion can be useful in the context of concretely producing many gauge-equivalent holomorphic structures on $\mathcal{L}_n$. Building on this framework, we introduce a one-parameter family of nonlinear maps that act on an appropriate Banach subspace of the Banach algebra $\mathscr{D}$. Finally, we obtain a necessary and sufficient criterion for the gauge equivalence of holomorphic structures on $\mathcal{L}_n$ in terms of the existence of fixed points of these maps (Thm. \ref{main}). This shows that the existence of an invertible solution to the gauge equation is intimately related to a fixed--point problem, highlighting that the situation of $\mathbb{C}\mathrm{P}_q^1$ is considerably more subtle than its classical analogue $\mathbb{C}\mathrm{P}^1$.

The paper is organized as follows. In Section~2, we briefly recall the essential preliminary material. In Section~3, we revisit the D\polhk{a}browski--Sitarz Dirac operator on the standard Podleś sphere and formulate the mixed gauge equation involving the anti-holomorphic differential $\bar\partial$. Section~4 contains the core of the paper, where we formulate the fixed--point--theoretic framework, and obtain a sufficient condition for the gauge equivalence. Finally, in Section~5 we derive, within this framework, a necessary and sufficient criterion for the gauge equivalence of holomorphic structures on $\mathcal{L}_n$ over $\mathcal{A}(\mathbb{C}\mathrm{P}_q^1)$ in terms of the existence of fixed points of certain maps acting on an appropriate Banach space.
\medskip


\section{The Hopf algebras $\mathcal{U}_q(\mathfrak{su}(2)),\,\mathcal{A}(SU_q(2))$, and $\mathbb{C}\mathrm{P}_q^1$}\label{sec:hopf-algebras}

Let $q$ be a real number with $0\,<\,q\,<\,1$. Throughout the article, for any $s\in\mathbb{N}$ the ``$q$-number'' $[s]_q=\frac{q^{-s}-q^s}{q^{-1}-q}$ will be simply denoted by $[s]$. In some formulae, the symbol $X^\pm$ appears. This indicates that the corresponding formula should be understood separately for $X$ and for $X^{-1}$.

In the following, we follow the notation of \cite{Klimyk}.

\subsection{The Hopf algebras $\mathcal{U}_q(\mathfrak{su}(2))\mbox{ and }\mathcal{A}(SU_q(2))$}

The quantum universal enveloping algebra $\mathcal{U}_q(\mathfrak{su}(2))$ is the Hopf $\ast$-algebra generated by $E,\,F,\,K,\,K^{-1}$ with relations $KK^{-1}\,=\,K^{-1}K\,=\,1$ and
\[
K^\pm E\,=\,q^\pm EK^\pm,\quad K^\pm F\,=\,q^\mp FK^\pm,\quad[E,F]\,=\,\frac{K^2-K^{-2}}{q-q^{-1}}.
\]
The $\ast$-structure is given by
\[
K^*\ =\ K,\quad E^*\ =\ F,\quad F^*\ =\ E,
\]
and the Hopf algebra structure is provided by coproduct $\Delta$, antipode $S$, and counit $\varepsilon$:
\[
\Delta(K^\pm)=K^\pm\otimes K^\pm,\quad \Delta(E)=E\otimes K+K^{-1}\otimes E,\quad \Delta(F)=F\otimes K+K^{-1}\otimes F,
\]
\[
S(K)\ =\ K^{-1},\quad S(E)\ =\ -qE,\quad S(F)\ =\ -q^{-1}F,
\]
\[
\varepsilon(K)\ =\ 1,\quad\varepsilon(E)\ =\ \varepsilon(F)\ =\ 0.
\]
On the other hand, the coordinate algebra $\mathcal{A}(SU_q(2))$ of the compact quantum group $SU_q(2)$ of Woronowicz \cite{Woronowicz} is the unital $\ast$-algebra generated by elements $a$ and $c$, with relations
\begin{eqnarray}\label{SUq(2)}
ac\ =\ qca,\quad ac^*\ =\ qc^*a,\quad cc^*\ =\ c^*c,\nonumber\\
a^*a+c^*c\ =\ 1,\quad aa^*+q^2cc^*\ =\ 1.
\end{eqnarray}
These relations are equivalent to requiring that the ``defining'' matrix
\[
U\ =\ \begin{pmatrix}
a & -qc^*\\
c & a^*
\end{pmatrix}
\]
is unitary. The Hopf $\ast$-algebra structure on $\mathcal{A}(SU_q(2))$ is given by the following:
\[
\Delta(U)\,=\,U\otimes U,\quad\Delta(U^*)\,=\,U^*\otimes U^*,\quad S(U)\,=\,U^*,\quad \varepsilon(U)\,=\,I_2.
\]


\subsection{The dual pairing}
There is a non-degenerate bilinear pairing between $\mathcal{U}_q(\mathfrak{su}(2))$ and $\mathcal{A}(SU_q(2))$, given
on the generators by
\[
\langle K^\pm,a\rangle=q^{\mp\frac{1}{2}},\quad\langle K^\pm,a^*\rangle=q^{\pm\frac{1}{2}},\quad\langle E,c\rangle=1,\quad\langle F,c^*\rangle=-q^{-1},
\]
and all other pairings of generators are zero. One regards $\mathcal{U}_q(\mathfrak{su}(2))$ as a subspace of
the linear dual of $\mathcal{A}(SU_q(2))$ via this pairing. Then, there are canonical left and right
$\mathcal{U}_q(\mathfrak{su}(2))$-module algebra structures on $\mathcal{A}(SU_q(2))$ such that
\[
\langle g,\, h\triangleright x\rangle\ :=\ \langle gh,\, x\rangle,\quad \langle g,\,x\triangleleft h\rangle
\ :=\ \langle hg,x\rangle
\]
for all $g,\, h\,\in\, \mathcal{U}_q(\mathfrak{su}(2))$ and $x\,\in\,\mathcal{A}(SU_q(2))$.
They are given by $h\triangleright x\, =\, \langle\mathrm{id}\otimes h,\,\Delta x\rangle$ and
$x\triangleleft h\,=\,\langle h\otimes\mathrm{id},\,\Delta x\rangle$, or equivalently,
\[
h\triangleright x\ =\ x_{(1)}\langle h,\,x_{(2)}\rangle,\quad x\triangleleft h\ =\
\langle h,\,x_{(1)}\rangle x_{(2)},
\]
in the Sweedler notation $\Delta(x)\,=\,x_{(1)}\otimes x_{(2)}$ for the coproduct. These left and right actions are related by the antipodes\,: $S(Sh\triangleright x)\,=\,Sx\triangleleft h$.

On the generators, these canonical
left and right actions are given by the following:
\[
K^\pm\triangleright a^s=q^{\mp\frac{s}{2}}a^s,\quad K^\pm\triangleright (a^*)^s=q^{\pm\frac{s}{2}}(a^*)^s,\quad K^\pm\triangleright c^s=q^{\mp\frac{s}{2}}c^s,\quad K^\pm\triangleright (c^*)^s=q^{\pm\frac{s}{2}}(c^*)^s
\]
\[
E\triangleright a^s=-q^{\frac{3-s}{2}}[s]a^{s-1}c^*\quad,\quad E\triangleright c^s=q^{\frac{1-s}{2}}[s]c^{s-1}a^*\quad,\quad E\triangleright (a^*)^s=E\triangleright(c^*)^s=0
\]
\[
F\triangleright (a^*)^s=q^{\frac{1-s}{2}}[s]c(a^*)^{s-1}\quad,\quad F\triangleright (c^*)^s=-q^{-\frac{1+s}{2}}[s]a(c^*)^{s-1}\quad,\quad F\triangleright a^s=F\triangleright c^s=0;
\]
and
\[
a^s\triangleleft K^\pm=q^{\mp\frac{s}{2}}a^s,\quad (a^*)^s\triangleleft K^\pm=q^{\pm\frac{s}{2}}(a^*)^s,\quad c^s\triangleleft K^\pm=q^{\pm\frac{s}{2}}c^s,\quad (c^*)^s\triangleleft K^\pm=q^{\mp\frac{s}{2}}(c^*)^s
\]
\[
(a^*)^s\triangleleft E=-q^{\frac{3-s}{2}}[s]a^{s-1}c^*\quad,\quad c^s\triangleleft E=q^{\frac{1-s}{2}}[s]c^{s-1}a^*\quad,\quad a^s\triangleleft E=(c^*)^s\triangleleft E=0
\]
\[
a^s\triangleleft F=q^{\frac{s-1}{2}}[s]ca^{s-1}\quad,\quad (c^*)^s\triangleleft F=-q^{\frac{s-3}{2}}[s]a^*(c^*)^{s-1}\quad,\quad (a^*)^s\triangleleft F=c^s\triangleleft F=0
\]
for any $s\in\mathbb{N}$, where $[s]$ denotes the $q$-number. These left and right actions are mutually commuting\,:
\[
(h\triangleright a)\triangleleft g\ =\ h\triangleright (a\triangleleft g),
\]
for all $h,g\in\mathcal{U}_q(\mathfrak{su}(2))$ and $a\in\mathcal{A}(SU_q(2))$, and since the pairing satisfies $\langle (Sh)^*,x\rangle\,=\,\overline{\langle h,x^*\rangle}$, the
$\ast$-structure is compatible with both the actions\,:
\[
h\triangleright x^*\ =\ \left((Sh)^*\triangleright x)\right)^*,\quad x^*\triangleleft h
\ =\ \left(x\triangleleft (Sh)^*)\right)^*.
\]
There is another left action, denoted by $\partial$, of $\mathcal{U}_q(\mathfrak{su}(2))$ on
$\mathcal{A}(SU_q(2))$, given by the formula:
\[
\partial_{g}(x)\ :=\ x\triangleleft S^{-1}(g)
\]
for $g\,\in\,\mathcal{U}_q(\mathfrak{su}(2))$ and $x\,\in\,\mathcal{A}(SU_q(2))$. Explicitly on the
generators, we have
\[
\partial_K(a)=q^{\frac{1}{2}}a,\quad\partial_K(a^*)=q^{-\frac{1}{2}}a^*,\quad\partial_K(c)=q^{-\frac{1}{2}}c,\quad\partial_K(c^*)=q^{\frac{1}{2}}c^*
\]
\[
\partial_E(a^*)=c^*\quad,\quad\partial_E(c)=-q^{-1}a\quad,\quad\partial_E(a)=\partial_E(c^*)=0
\]
\begin{equation}\label{reqd actions}
\partial_F(a)=-qc\quad,\quad\partial_F(c^*)=a^*\quad,\quad\partial_F(c)=\partial_F(a^*)=0.
\end{equation}
This action is used in \cite{Schmudgen-Wagner}.


\subsection{The homogeneous space $\mathcal{A}(S_q^2)$}
The principal bundle structure on the compact quantum group $SU_q(2)$ is
constructed using the right $U(1)$-action $\alpha\,:\,U(1)\, \longrightarrow\,
\mathrm{Aut}(\mathcal{A}(SU_q(2)))$, which on the generators is given by $a\,\mapsto\, za$ and $c\,\mapsto\, zc$ for $z\,\in\, U(1)$. We have a vector space decomposition $\mathcal{A}(SU_q(2))
\,=\,\oplus_{n\in\mathbb{Z}}\,\mathcal{L}_n$, where
\[
\mathcal{L}_n\ :=\ \{x\,\in\,\mathcal{A}(SU_q(2))\,\,\big\vert\,\, \alpha_z(x)
\,=\,z^{-n}x\ \,\,\forall\,z\,\in\,  U(1)\}.
\]
Alternatively, we have
\begin{eqnarray}\label{bimod Ln}
\mathcal{L}_n\ =\ \{x\,\in\,\mathcal{A}(SU_q(2))\,\,\big\vert\,\, K\triangleright x
\, =\, q^{\frac{n}{2}}x\}.
\end{eqnarray}
Observe that $\mathcal{L}_n\mathcal{L}_m\,\subseteq\, \mathcal{L}_{n+m}$ and $\mathcal{L}_n^*
\,\subseteq\, \mathcal{L}_{-n}$.

We may also consider the left $U(1)$-action $\beta\,:\,U(1)\,\longrightarrow\, \mathrm{Aut}(\mathcal{A}
(SU_q(2)))$, given on the generators by $a\,\longmapsto\, za$ and $c\,\longmapsto\, \overline{z}c$
for $z\,\in\, U(1)$. We obtain a vector space decomposition $\mathcal{A}(SU_q(2))
\,=\,\oplus_{n\in\mathbb{Z}}\,M_n$, where
\[
M_n\ :=\ \{x\,\in\,\mathcal{A}(SU_q(2))\,\,\big\vert\,\, \alpha_z(x)\,=\,z^nx\ \,\,\forall\,
\, z\,\in\, U(1)\}.
\]
Alternatively, we have
\begin{eqnarray}\label{bimod Mn}
M_n &=&\{x\,\in\,\mathcal{A}(SU_q(2))\,\,\big\vert\,\, \partial_K(x)\,=\, q^{-\frac{n}{2}}x\}\nonumber\\
&=& \{y\,\in\,\mathcal{A}(SU_q(2))\,\,\big\vert\,\, y\triangleleft K\,=\,q^{\frac{n}{2}}y\}\,,
\end{eqnarray}
(recall the $\partial_K$-action from \eqref{reqd actions}). This formulation has been used in \cite{Schmudgen-Wagner}. Precise relationship between these two decompositions of $\mathcal{A}(SU_q(2))$ is given by the following.

\begin{lemma}\label{essentially same}
The following holds: $$S(\mathcal{L}_n)\ =\ M_{-n}$$ for each $n\,\in\,\mathbb{Z}$, where $S$ is the antipode
of $\,\mathcal{A}(SU_q(2))$.
\end{lemma}\label{lem:antipode-Ln-Mn}

\begin{proof}
Recall that the antipode $S$ in a Hopf $\ast$-algebra is invertible with inverse $S^{-1}=\ast\circ S\circ\ast$ \cite[Section 1.2.7]{Klimyk}. We have the following if and only if statements\,:
\begin{eqnarray*}
Sx\in M_n &\iff& Sx\triangleleft K=q^{\frac{n}{2}}Sx\\
&\iff& S(S(K)\triangleright x)=q^{\frac{n}{2}}Sx\\
&\iff& S(K^{-1}\triangleright x)=q^{\frac{n}{2}}Sx\\
&\iff& K^{-1}\triangleright x=q^{\frac{n}{2}}x\qquad[\mbox{since }S\mbox{ is invertible}]\\
&\iff& K\triangleright x=q^{-\frac{n}{2}}x\\
&\iff& x\in L_{-n}
\end{eqnarray*}
for all $x\in\mathcal{A}(SU_q(2))$.
\end{proof}
Note that for each $n\in\mathbb{Z}$,
\[
F\triangleright \mathcal{L}_n\subseteq \mathcal{L}_{n-2},\quad E\triangleright \mathcal{L}_n\subseteq \mathcal{L}_{n+2},\quad K\triangleright \mathcal{L}_n\subseteq \mathcal{L}_n
\]
whereas
\[
\partial_F(M_n)\subseteq M_{n+2},\quad \partial_E(M_n)\subseteq M_{n-2},\quad \partial_K(M_n)\subseteq M_n.
\]
We also have $\mathcal{L}_n\triangleleft g\subseteq \mathcal{L}_n$ and $g\triangleright M_n\subseteq M_n$, for all $g\in\mathcal{U}_q(\mathfrak{su}(2))$.

The fixed point subalgebra under the right $U(1)$-action (left $K\triangleright\,\cdot\,$-action) is the coordinate algebra $\mathcal{A}(S_q^2)$ of the standard Podle\'s sphere \cite{Podles1987QuantumSpheres} generated by
\begin{eqnarray}\label{Sq^2}
B_0:=cc^*\,,\quad B_-:=ac^*\,,\quad B_+:=ca^*
\end{eqnarray}
subject to the following relations:
\begin{eqnarray}\label{Sq^21}
B_-B_0=q^2B_0B_-,\quad B_+B_-=B_0(1-B_0),\quad B_-B_+=q^2B_0(1-q^2B_0).
\end{eqnarray}
All the $\mathcal{L}_n$'s in \eqref{bimod Ln} are bimodules over $\mathcal{A}(S_q^2)$, and $\mathcal{L}_0=\mathcal{A}(S_q^2)$. It is known that $\mathcal{L}_n$'s are finitely generated projective $\mathcal{A}(S_q^2)$-modules,
analogous to the canonical line bundles $\mathcal{O}(-n)$ on $\mathbb{C}\mathrm{P}^1$ of degree
$-n\, \in\, {\mathbb Z}$.

Alternatively, we can also use the left $U(1)$-action (the $\partial_K$-action), and the fixed point subalgebra now has the following alternative description: unital $\ast$-subalgebra of $\mathcal{A}(SU_q(2))$ generated by
\begin{eqnarray}\label{alternative Sq^2}
A:=c^*c\,,\quad B:=ac\,,\quad B^*:=c^*a^*
\end{eqnarray}
subject to the following relations:
\begin{eqnarray}\label{alternative Sq^21}
BA\ =\ q^2AB,\quad B^*B\ =\ A-A^2,\quad BB^*\ =\ q^2A(1-q^2A).
\end{eqnarray}
All the $M_n$'s in \eqref{bimod Mn} are bimodules over $\mathcal{A}(S_q^2)$, and $M_0\,=\, \mathcal{A}(S_q^2)$.


\subsection{The $C^*$-algebras $C(SU_q(2))\mbox{ and }C(S_q^2)$}

The algebra $\mathcal{A}(SU_q(2))$ has a vector space basis consisting of matrix elements of its
irreducible corepresentations $$\{t^\ell_{m,n}\,\, \big\vert\,\, \ell\,\in\,\frac{1}{2}\mathbb{N},
\ \,m,\,n\,=\,-\ell,\,-\ell+1,\,\cdots,\, \ell\}.$$ In particular, we have
\[
t^0_{00}\ =\ 1,\quad t^{\frac{1}{2}}_{-\frac{1}{2},-\frac{1}{2}}\ =\ 
a,\quad t^{\frac{1}{2}}_{\frac{1}{2},-\frac{1}{2}}\ =\ c.
\]
The coproduct satisfies the formula $\Delta(t^\ell_{mn})\ =\ \sum_k\,t^\ell_{mk}\otimes t^\ell_{kn}$, and the
product satisfies the formula
\[
t^j_{rs}\otimes t^\ell_{mn}\ =\ \sum_{}^{}\,C_q\begin{pmatrix}
j & \ell & k\\
r & m & r+m
\end{pmatrix}\,C_q\begin{pmatrix}
j & \ell & k\\
s & n & s+n
\end{pmatrix}\,t^k_{r+m,s+n},
\]
where the $C_q$'s are the $q$-Clebsch-Gordan coefficients.

The universal $C^*$-algebra $C(SU_q(2))$, which is generated by $a,\,c$ subject to the relations same as 
in $\mathcal{A}(SU_q(2))$, is the first example of a compact quantum group \cite{Woronowicz}. One of 
the main features of this compact quantum group is the existence of a unique left-invariant Haar state that we 
shall denote by $h$. This state is faithful, and $h(t^\ell_{mn})\,=\,0$ for $\ell\,>\,0$. Note that the $C^*$-algebras 
$C(SU_q(2))$ for different $q\,\in\,(0,\,1)$ are all isomorphic, although as compact quantum groups this is 
not the case (recall that we have $0\,<\,q\,<\,1$).

Let $L^2(SU_q(2))$ denote the GNS-Hilbert space with respect to the Haar state with an orthonormal 
basis (the Peter-Weyl basis) 
$$\left\{|\ell,\,m,\,n\rangle\,:=\,q^m[2\ell+1]^{\frac{1}{2}}\eta(t^\ell_{m,n})\,\,\big\vert\,\,
\ell\,\in\,\frac{1}{2}\mathbb{N},\,\, m,\,n\,=\,-\ell,\,-\ell+1,\, \cdots,\,\ell\right\},$$ 
where $\eta\,:\,C(SU_q(2))\,\longrightarrow\, L^2(SU_q(2))$ is the GNS map which is injective.

We have the following irreducible representation of $\mathcal{U}_q(\mathfrak{su}(2))$ on the $(2\ell+1)$-dimensional
vector space $V_\ell\,:=\,\mathrm{span}\{|\ell,\,m,\,n\rangle\,\,\big\vert\,\,
m\,=\, -\ell,\,-\ell+1,\,\cdots,\,\ell\}$ given by
\begin{align*}
\sigma(K)|\ell,m,n\rangle &= q^n|\ell,m,n\rangle\\
\sigma(E)|\ell,m,n\rangle &= ([\ell-n][\ell+n+1])^{\frac{1}{2}}|\ell,m,n+1\rangle\\
\sigma(F)|\ell,m,n\rangle &= ([\ell+n][\ell-n+1])^{\frac{1}{2}}|\ell,m,n-1\rangle.
\end{align*}
This representation is such that the left regular representation $\pi$ of $C(SU_q(2))$ on the GNS-space $L^2(SU_q(2))$ is equivariant with respect to the left $\mathcal{U}_q(\mathfrak{su}(2))$-action:
\[
\sigma(g)\pi(f)\ =\ \pi(g\triangleright f),\quad g\,\in\, \mathcal{U}_q(\mathfrak{su}(2)),\,\, \,f
\,\in \,C(SU_q(2)).
\]
The right $U(1)$-action $\alpha$ on $\mathcal{A}(SU_q(2))$ defined in Section $2.3$ extends to $C(SU_q(2))$ by
automorphisms, and we have a $C^*$-dynamical system $\left(C(SU_q(2)),\,L^2(SU_q(2)),\,
U(1),\,\alpha\right)$. The invariant $\ast$-subalgebra in $C(SU_q(2))$ under this action is by definition the $C^*$-algebra of the standard Podle\'s sphere, denoted by $C(S_q^2)\,
:=\,C(SU_q(2))^{U(1)}$. It is well-known that $C(S_q^2)\,\cong\,
\mathcal{K}(\ell^2(\mathbb{N}))\oplus\mathbb{C}$ (the minimal unitization of compact operators), where
$\mathcal{K}(\ell^2(\mathbb{N}))$ denotes the space of all compact operators on $\ell^2(\mathbb{N})$.

It turns out that the restriction of the GNS-map $\eta\,:\,C(S_q^2)\,\longrightarrow\,
L^2(SU_q(2))^{U(1)}$ is still injective, and we have
\[
L^2(S_q^2)\,:=\,L^2(C(S_q^2))\,=\,\overline{\mathrm{span}}\left\{|\ell,\,m,\,0\rangle\,\,\big\vert\,
\, \ell\,\in\,\frac{1}{2}\mathbb{N},\,\,m\,=\,-\ell,\,-\ell+1,\,\cdots,\,\ell\right\},
\]
which is a closed subspace of $L^2(SU_q(2))$.


\subsection{$\mathbb{C}\mathrm{P}_q^1$ and holomorphic line bundles}\label{subsec:holomorphic-line-bundles}

The quantum principal $U(1)$-bundle $\mathcal{A}(S_q^2)\, \xhookrightarrow{}\,\mathcal{A}(SU_q(2))$ is endowed
(\cite{BM1}, \cite{BM2}) with compatible nonuniversal calculi obtained from Woronowicz's left-covariant $3D$-calculus on $\mathcal{A}(SU_q(2))$ \cite{Woronowicz}. Its restriction on $\mathcal{A}(S_q^2)$
gives the unique left-covariant $2D$-calculus on $\mathcal{A}(S_q^2)$ (\cite{Podles2}, \cite{Majid}). This unique
calculus was also realized via a Dirac operator on $\mathcal{A}(S_q^2)$ \cite{Schmudgen-Wagner}. We describe it briefly here and refer to Section $3.4$ in \cite{KhalkhaliLandiVanSuijlekom2011} for details. We set
\[
\omega_-\ =\ c^*da^*-qa^*dc^*\,,\quad\omega_+\ =\ adc-qcda,
\]
\[
X_-\ =\ q^{-\frac{1}{2}}FK\,,\quad X_+\ =\ q^{\frac{1}{2}}EK,
\]
where $E,F,K\in\mathcal{U}_q(\mathfrak{su}(2))$. Then, for $f\,\in\,\mathcal{A}(S_q^2)$ we have the differential
\[
df\ =\ (X_-\triangleright f)\omega_-+(X_+\triangleright f)\omega_+,
\]
where the holomorphic and the anti-holomorphic differentials are respectively
\begin{eqnarray}\label{the differentials}
\partial f\,=\,(X_+\triangleright f)\,\omega_+\,,\quad\overline\partial f\,=\,(X_-\triangleright f)\,\omega_-.
\end{eqnarray}
They satisfy the condition $(\partial f)^*\,=\,\overline\partial f^*$. This decomposes $\Omega^1(S_q^2)$ into $\Omega^{1,0}(S_q^2)\oplus\Omega^{0,1}(S_q^2)$, and we also have $\Omega^2(S_q^2)=\Omega^{1,1}(S_q^2)$ with $\Omega^{2,0}(S_q^2)=\Omega^{0,2}(S_q^2)=\{0\}$. All the higher forms vanish. Thus,
\[
\Omega^\bullet(S_q^2)=\mathcal{A}(S_q^2)\oplus\big(\Omega^{1,0}(S_q^2)\oplus\Omega^{0,1}(S_q^2)\big)\oplus\Omega^{1,1}(S_q^2)\,.
\]
The splitting $\Omega^1(S_q^2)=\Omega^{1,0}(S_q^2)\oplus\Omega^{0,1}(S_q^2)$ together with the differentials $\partial\mbox{ and }\bar\partial$ given in \eqref{the differentials} constitute a complex structure for the differential calculus on the standard Podle\'s sphere \cite[Proposition 3.4]{KhalkhaliLandiVanSuijlekom2011}. We remark that $\Omega^{1,0}(S_q^2)\mbox{ and }\Omega^{0,1}(S_q^2)$ are rank one finitely generated projective bimodules and $\Omega^{1,1}(S_q^2)$ is a free rank one bimodule over $\mathcal{A}(S_q^2)$.
\smallskip

\noindent\textbf{Notation:} Following \cite{KhalkhaliLandiVanSuijlekom2011}, henceforth we shall write $\mathcal{A}(\mathbb{C}\mathrm{P}_q^1)$ in place of $\mathcal{A}(S_q^2)$ to denote the coordinate algebra of the quantum projective line $\mathbb{C}\mathrm{P}_q^1\,$. This is to stress on the associated complex structure. Similarly, we shall write $C(\mathbb{C}\mathrm{P}_q^1)$ in place of $C(S_q^2)$ for the $C^*$-algebra.

\begin{definition}\label{holomorphic vb}
A finitely generated projective left module $E$ over $\mathcal{A}(\mathbb{C}\mathrm{P}_q^1)$ is called holomorphic
if there exists a flat $\overline\partial$-connection $\overline{\nabla}_E\,:\,E\,\longrightarrow\,
\Omega^{0,1}(\mathbb{C}\mathrm{P}_q^1)\otimes_{\mathcal{A}(\mathbb{C}\mathrm{P}_q^1)}E$. The pair $(E,\,\overline{\nabla}_E)$ is called a holomorphic vector bundle over the quantum projective line $\mathbb{C}\mathrm{P}_q^1$.
\end{definition}

Each bimodule $\mathcal{L}_n$ in \eqref{bimod Ln} is a rank-one finitely generated projective left $\mathcal{A}(\mathbb{C}\mathrm{P}_q^1)$-module, and enjoy the following properties
\begin{eqnarray}\label{properties}
\mathcal{L}_n^*=\mathcal{L}_{-n}\,,\qquad\mathcal{L}_n\otimes_{\mathcal{A}(\mathbb{C}\mathrm{P}_q^1)}\mathcal{L}_m\cong \mathcal{L}_{n+m}
\end{eqnarray}
for any $n,m\in\mathbb{Z}$, where the isomorphism is the multiplication map \cite[Proposition 3.1]{KhalkhaliLandiVanSuijlekom2011}. We have $\Omega^{0,1}(\mathbb{C}\mathrm{P}_q^1)=\mathcal{L}_{-2}\,\omega_-$ and $\Omega^{1,0}(\mathbb{C}\mathrm{P}_q^1)=\mathcal{L}_2\,\omega_+$. Each $\mathcal{L}_n$ carries a flat $\overline\partial$-connection $\overline{\nabla}^{(n)}_0\,:\,
\mathcal{L}_n\,\rightarrow\,\Omega^{0,1}(\mathbb{C}\mathrm{P}_q^1)\otimes_{\mathcal{A}(\mathbb{C}\mathrm{P}_q^1)}\mathcal{L}_n$ that induces a holomorphic structure on it, and it is defined by
\begin{eqnarray}\label{standard hol}
\overline{\nabla}^{(n)}_0\phi\ :=\ q^{-n+2}\omega_-(X_-\triangleright\phi).
\end{eqnarray}
In particular, $\overline\partial\,:\,\mathcal{L}_0\,=\,\mathcal{A}(\mathbb{C}\mathrm{P}_q^1)\,\longrightarrow\,
\Omega^{0,1}(\mathbb{C}\mathrm{P}_q^1)\cong\mathcal{L}_{-2}$ is the holomorphic structure $\overline{\nabla}^{(0)}_0$ on the trivial line bundle $\mathcal{L}_0$. It is known that $\mathrm{ker}(\bar\partial)\,=\,\mathbb{C}$ \cite[Section 4.1]{KhalkhaliLandiVanSuijlekom2011}.

The connection $\overline{\nabla}^{(n)}_0$ will be referred to as the standard holomorphic structure on $\mathcal{L}_n$ throughout the article.

At the $C^*$-algebraic level, we can use the $U(1)$-action to decompose $C(SU_q(2))$ into
$C(\mathbb{C}\mathrm{P}_q^1)$-modules
\[
\Gamma(\mathcal{L}_n)\ =\ \{f\,\in\, C(SU_q(2))\,\, \big\vert\,\, \alpha_z(f)
\,=\,z^{-n}f \, \,\, \forall\,\, z\,\in\, U(1)\},
\]
which are spaces of continuous sections on the line bundles $\mathcal{L}_n$. The space of $L^2$-sections are defined as
\[
L^2(\mathcal{L}_n)\ :=\ \{\psi\,\in\, L^2(SU_q(2))\,\,\big\vert\,\,
\rho(z)\psi\,=\,z^{-n}\psi\,\, \,\forall\,\,z\,\in\, U(1)\}.
\]
The operator $\sigma(F)$ is unbounded acting on $L^2(SU_q(2))$ with a dense domain that contains the
image of $\mathcal{A}(\mathbb{C}\mathrm{P}_q^1)$ inside $L^2(\mathbb{C}\mathrm{P}_q^1)\,\subseteq\, L^2(SU_q(2))$ under the
GNS-map $\eta\,:\,C(\mathbb{C}\mathrm{P}_q^1)\,\longrightarrow\, L^2(\mathbb{C}\mathrm{P}_q^1)$. It turns out that the kernel of $\sigma(F)$
restricted to $L^2(\mathbb{C}\mathrm{P}_q^1)$ is $\mathbb{C}$, and as a consequence, there are no nontrivial holomorphic
polynomial functions on $\mathbb{C}\mathrm{P}_q^1$. Furthermore, there are no nontrivial holomorphic functions
in $\mathrm{Dom}(\overline\partial)\cap C(\mathbb{C}\mathrm{P}_q^1)$, where
\[
\mathrm{Dom}(\overline\partial)\ :=\ \{f\,\in\, C(SU_q(2))\,\,\big\vert\,\,
\norm{F\triangleright f}\,<\,\infty\}.
\]
See \cite[Section 4]{KhalkhaliLandiVanSuijlekom2011} for details.
\medskip


\section{The Dirac operator for $\mathbb{C}\mathrm{P}_q^1$ and the Gauge equation}


\subsection{The gauge equation and the mixed gauge equation}

Let $\mathscr{A}$ be a noncommutative algebra equipped with a complex structure in the sense of \cite[Definition 2.1]{KhalkhaliLandiVanSuijlekom2011}, and $E$ be a holomorphic vector bundle over $\mathscr{A}$. That is, $E$ is a finitely generated projective (f.g.p in short) left module over $\mathscr{A}$ and there is a flat $\overline\partial$-connection $\overline\nabla\,:\,E\,\longrightarrow\,\Omega^{0,1}(\mathscr{A})\otimes_{\mathscr{A}} E$. Two flat $\overline\partial$-connections $\overline\nabla_j$, for $j\,=\,1,\,2$, are called gauge equivalent if there exists an invertible element $g\,\in\,\mathcal{E}nd_{\mathscr{A}}(E)$ such that $g\overline\nabla_2g^{-1}
\,=\,\overline\nabla_1$.

In this article, we are interested in $\mathscr{A}:=\mathcal{A}(\mathbb{C}\mathrm{P}_q^1)$ and $E\,:=\,\mathcal{L}_n$. In this case, we have $\mathcal{E}nd_{\mathcal{A}(\mathbb{C}\mathrm{P}_q^1)}(\mathcal{L}_n)\,=\,
\mathcal{A}(\mathbb{C}\mathrm{P}_q^1)$, realized through the right multiplication. Recall that for each $n\,\in\,\mathbb{Z}$, there is a twisted flip isomorphism
\begin{equation}\label{twisted}
\Phi_{(n)}\ :\ \mathcal{L}_n\otimes_{\mathcal{A}(\mathbb{C}\mathrm{P}_q^1)} \Omega^{0,1}(\mathbb{C}\mathrm{P}_q^1)\
\longrightarrow\ \Omega^{0,1}(\mathbb{C}\mathrm{P}_q^1)\otimes_{\mathcal{A}(\mathbb{C}\mathrm{P}_q^1)} \mathcal{L}_n
\end{equation}
of $\mathcal{A}(\mathbb{C}\mathrm{P}_q^1)$-bimodules \cite[Lemma 3.6]{KhalkhaliLandiVanSuijlekom2011}. Then, the difference between
two $\bar\partial$-connections on $\mathcal{L}_n$ is an element in
\begin{eqnarray}\label{230}
\mathrm{Hom}_{\mathcal{A}(\mathbb{C}\mathrm{P}_q^1)}(\mathcal{L}_n,\,\Omega^{0,1}(\mathbb{C}\mathrm{P}_q^1)\otimes_{\mathcal{A}(\mathbb{C}\mathrm{P}_q^1)} \mathcal{L}_n)\ \cong\ \Omega^{0,1}(\mathbb{C}\mathrm{P}_q^1)\,,
\end{eqnarray}
realized through the right multiplication. Since $\overline{\nabla}^{(n)}_0$, given in \eqref{standard hol}, is a flat connection on $\mathcal{L}_n$, we have the cochain complex
\begin{eqnarray}\label{231}
0\,\,\overset{0}{\longrightarrow}\,\,\mathcal{L}_n\,\,\overset{\overline{\nabla}^{(n)}_0}{\longrightarrow}\,\,\Omega^{0,1}(\mathbb{C}\mathrm{P}_q^1)\otimes_{\mathcal{A}(\mathbb{C}\mathrm{P}_q^1)}\mathcal{L}_n\,\,\overset{\overline{\nabla}^{(n)}_0}{\longrightarrow}\,\, 0
\end{eqnarray}
(because $\Omega^{0,2}(\mathbb{C}\mathrm{P}_q^1)=\{0\}$) in which
\[
H^{0,0}_{\overline{\nabla}^{(n)}_0}(\mathcal{L}_n)\ =\ \begin{cases}
0\,,\qquad\,n>0,\\
\mathbb{C}^{|n|+1},\,\,n\leq 0;
\end{cases}
\]
(see \cite[Theorem 4.4]{KhalkhaliLandiVanSuijlekom2011}). For $n=0$, we have in particular, $\overline{\nabla}^{(0)}_0=\bar\partial$ on $\mathcal{L}_0=\mathcal{A}(\mathbb{C}\mathrm{P}_q^1)$, and
\begin{eqnarray}\label{2321}
\Omega^{0,1}(\mathbb{C}\mathrm{P}_q^1)\otimes_{\mathcal{A}(\mathbb{C}\mathrm{P}_q^1)}\mathcal{L}_0\cong\Omega^{0,1}(\mathbb{C}\mathrm{P}_q^1)\,.
\end{eqnarray}
The cochain complex in \eqref{231} reduces to the Dolbeault complex\,:
\begin{eqnarray}\label{232}
0\,\,\overset{0}{\longrightarrow}\,\,\mathcal{L}_0\,\,\overset{\bar\partial}{\longrightarrow}\,\,\Omega^{0,1}(\mathbb{C}\mathrm{P}_q^1)\,\,\overset{\bar\partial}{\longrightarrow}\,\,\Omega^{0,2}(\mathbb{C}\mathrm{P}_q^1)=0
\end{eqnarray}
and we have $H^{0,1}_{\overline\partial}(\mathcal{L}_0)\,=\,0$
\cite[Proposition 7.2]{DAndreaLandi2013GeometryQuantumProjectiveSpaces}. Therefore, from \eqref{232} we conclude that the map
$\overline\partial\,:\,\mathcal{L}_0=\mathcal{A}(\mathbb{C}\mathrm{P}_q^1)\,\longrightarrow\, \Omega^{0,1}(\mathbb{C}\mathrm{P}_q^1)$ is surjective. This fact, together with \eqref{230}, says that any
element $\theta\,\in\, \mathrm{Hom}_{\mathcal{A}(\mathbb{C}\mathrm{P}_q^1)}(\mathcal{L}_n,\,\Omega^{0,1}(\mathbb{C}\mathrm{P}_q^1)
\otimes_{\mathcal{A}(\mathbb{C}\mathrm{P}_q^1)}\mathcal{L}_n)$ is of the form $\overline\partial f$ for
some $f\,\in\,\mathcal{A}(\mathbb{C}\mathrm{P}_q^1)$. Moreover, $f$ is unique up to scalars as $\mathrm{ker}(\overline\partial)
\,=\, \mathbb{C}$.  Thus, all holomorphic structures on $\mathcal{L}_n$ are given by the following:
\begin{equation}\label{all connection}
\overline{\nabla}^{(n)}_{\overline\partial f}(s)\ :=\ \overline{\nabla}^{(n)}_0(s)-\Phi_{(n)}(s\otimes\overline\partial f)
\end{equation}
for some $f\,\in\,\mathcal{A}(\mathbb{C}\mathrm{P}_q^1)$, where $\overline{\nabla}^{(n)}_0$ is the standard holomorphic structure on $\mathcal{L}_n$ given in \eqref{standard hol}, and $\Phi_{(n)}$ as in \eqref{twisted}.

We remark that the standard $\bar\partial$-connection $\overline{\nabla}^{(n)}_0$ on $\mathcal{L}_n$ is in fact a $\Phi_{(n)}$-bimodule connection, that is, it satisfies the twisted right Leibniz property:
\begin{equation}\label{bimod con}
\overline{\nabla}^{(n)}_0(sx)=\overline{\nabla}^{(n)}_0(s)x+\Phi_{(n)}(s\otimes\bar\partial x)
\end{equation}
for $s\in\mathcal{L}_n$ and $x\in\mathcal{A}(\mathbb{C}\mathrm{P}_q^1)$. Each $\overline{\nabla}^{(n)}_{\bar\partial f}$ in \eqref{all connection} is a left $\bar\partial$-connection on $\mathcal{L}_n$, but unlike $\overline{\nabla}^{(n)}_0$, it is not a $\Phi_{(n)}$-bimodule connection on $\mathcal{L}_n$ in general. Now, recall the following result from \cite{GravemanLaRueMacArthurPesinWei2025NonStandardHolomorphic}.

\begin{lemma}\cite[Lemma 3.1]{GravemanLaRueMacArthurPesinWei2025NonStandardHolomorphic}\label{their gauge}
Two $\overline\partial$-connections $\overline{\nabla}_1^n,\, \overline{\nabla}^{(n)}_2$ on the line bundle $\mathcal{L}_n$ are gauge equivalent, that is $g\overline{\nabla}_2^n g^{-1}\,=\, \overline{\nabla}_1^n$, if and only if $\overline\partial g\,=\,g\overline\partial f_1-\overline\partial f_2\,.\,g$ for $f_1,\,f_2\,\in\,\mathcal{A}(\mathbb{C}\mathrm{P}_q^1)$ associated, via \eqref{all connection}, to $\overline{\nabla}_1^n\mbox{ and }\overline{\nabla}_2^n$ respectively.
\end{lemma}

\begin{definition}\label{mixed gauge equation}
The equation $\overline\partial g\,=\,g\overline\partial f$ is called the gauge equation/noncommutative
exponential equation, and the equation $\overline\partial g\,=\,g\overline\partial f-\overline\partial h\,.\,g$
is called the mixed gauge equation.
\end{definition}

Therefore, the problem of determining whether two $\bar\partial$-connections $\overline{\nabla}^{(n)}_{\overline\partial f}$ and $\overline{\nabla}^{(n)}_{\overline\partial h}$ as in \eqref{all connection} on line bundle $\mathcal{L}_n$ are gauge equivalent is to precisely solve the mixed gauge equation $\,\overline\partial g\,=\,g\overline\partial f-\overline\partial h\,.\,g\,$ for invertible $g\in\mathcal{E}nd_{\mathcal{A}(\mathbb{C}\mathrm{P}_q^1)}(\mathcal{L}_n)\cong\mathcal{A}(\mathbb{C}\mathrm{P}_q^1)$, with $f,\,h\in\mathcal{A}(\mathbb{C}\mathrm{P}_q^1)$. In the classical case, given any $f,\,h\,\in\, C^\infty(\mathbb{C}\mathrm{P}^1)$ we always have an invertible solution $g\,=\,\exp(f-h)$ of the gauge equation. However, lack of `many' invertible elements in $\mathcal{A}(\mathbb{C}\mathrm{P}_q^1)$, and considering the fact that even in the classical case solutions lie in the algebra of smooth functions rather than in the coordinate algebra, it is clear that we need to enlarge the domain of `$g$' and bring the $C^*$-algebra $C(\mathbb{C}\mathrm{P}_q^1)$ in the picture (also suggested in 
\cite[Section 3.1]{GravemanLaRueMacArthurPesinWei2025NonStandardHolomorphic}) to look for \emph{invertible} solutions of the mixed gauge equation. We shall formally formulate this in Section $3.3$ (see Defn. \ref{main definition}) as it requires some preparation.


\subsection{The Dirac operator and the anti-holomorphic differential}

Consider the $C^*$-completion $C(SU_q(2))$ of $\mathcal{A}(SU_q(2))$, the underlying $C^*$-algebra of the compact quantum group $SU_q(2)$. Let $h\,:\, C(SU_q(2))\,\rightarrow\, \mathbb{C}$ denote the Haar state, which is faithful in this case, and $L^2(SU_q(2))$ be the GNS Hilbert space with the inner-product given by
\[
  \langle x\,,\, y \rangle\ :=\ h(x^* y)\,.
\]
Recall from \eqref{bimod Mn} the decomposition
$\mathcal{A}(SU_q(2))\,=\,\bigoplus_{n\in\mathbb{Z}}\,M_n$ with $M_0\,=\,\mathcal{A}(\mathbb{C}\mathrm{P}_q^1)$. Set
\[
\mathcal{H}_+\ =\ \overline{M_1}^{\langle \cdot, \cdot \rangle},  \qquad 
  \mathcal{H}_-\ =\ \overline{M_{-1}}^{\langle \cdot, \cdot \rangle}.
\]
The restriction of the GNS-representation $\rho:C(SU_q(2))\to\mathcal{B}(L^2(SU_q(2)))$ to the standard Podle\'s sphere induces a faithful representation $\pi$ of the $C^*$-algebra $C(\mathbb{C}\mathrm{P}_q^1)$ on the Hilbert space $\mathcal{H}_+\oplus \mathcal{H}_-\,,$
\begin{eqnarray}\label{representation}
\pi(x)\ =\ \begin{pmatrix}
\pi_+(x) & 0\\
0 & \pi_-(x)
\end{pmatrix}\ :\ \mathcal{H}_+\oplus \mathcal{H}_-\ \longrightarrow\ \mathcal{H}_+\oplus \mathcal{H}_-\,.
\end{eqnarray}
where $\pi_\pm(x):=\rho(x)|_{\mathcal{H}_\pm}$. Note that $\pi_\pm$ are not faithful, but their direct sum $\pi$ is indeed so. We have $\norm{\pi(x)}_{\mathrm{op}}\,=\,\max_\pm\norm{\pi_\pm(x)}_{\mathrm{op}}$ for all $x\,\in\,C(\mathbb{C}\mathrm{P}_q^1)$. 

Now we recall the $0^+$-summable Dirac operator on the standard Podle\'s sphere introduced by D\polhk{a}browski and Sitarz \cite{Dabrowski-Sitarz}. We follow the formulation of Schm\"udgen and Wagner \cite{Schmudgen-Wagner}, where the Dirac operator is unitarily equivalent to that constructed by D\polhk{a}browski and Sitarz. We shall briefly recall it here, and refer to \cite[Section 3]{AguilarKaad2018PodlesSpectralMetric} or \cite[Section 2.2]{Kaad-Kyed} for details. Recall from \eqref{reqd actions} the linear maps
\[
\partial_E,\,\partial_F,\,\partial_K:\mathcal{A}(SU_q(2))\to\mathcal{A}(SU_q(2))
\]
where, $E,\,F,\,K$ are the generators of $\mathcal{U}_q(\mathfrak{su}(2))$. Using the dual pairing, it follows that $\partial_K$ is an algebra automorphism and that $\partial_E,\,\partial_F$ are twisted derivations, in the
sense that
\begin{align*}
\partial_E(xy) &= \partial_E(x)\partial_K(y)+\partial_K^{-1}(x)\partial_E(y)\\
\partial_F(xy) &= \partial_F(x)\partial_K(y)+\partial_K^{-1}(x)\partial_F(y)
\end{align*}
for all $x,y\in\mathcal{A}(SU_q(2))$. The $0^+$--summable spectral triple on the standard Podle\'s sphere is given by the following Dirac operator
\[
  \mathcal{D}\ =\ 
  \begin{pmatrix} 
    0 & \partial_F \\ 
    \partial_E & 0 
  \end{pmatrix}
\]
(see \cite[Section 2.2]{Kaad-Kyed} for details). This Dirac operator is not {\em isospectral} with the spin Dirac operator on the sphere $S^2$ equipped with the round metric, and it is a genuine quantum phenomenon. Note that for $x\,\in\,\mathcal{A}(\mathbb{C}\mathrm{P}_q^1)$, one has
\begin{eqnarray}\label{closable}
  [\mathcal{D}, x]\ =\ 
  \begin{pmatrix} 
    0 & \delta_2(x)\\ 
    \delta_1(x) & 0 
  \end{pmatrix}
\  \in\ \mathcal{B}(\mathcal{H}_+\oplus \mathcal{H}_-),
\end{eqnarray}
where $\delta_1\,=\,q^{\frac{1}{2}}\partial_E$ and $\delta_2\,=\,q^{-\frac{1}{2}}\partial_F$, and note that the derivation $d:\mathcal{A}(\mathbb{C}\mathrm{P}_q^1)\to\mathcal{B}(\mathcal{H}_+\oplus \mathcal{H}_-)$ induced by the commutator $[\mathcal{D},\cdot]$ is closable. The closable derivations $\delta_1$ and $\delta_2$ are related by the formula $\delta_1(x)^*\,=\,-\delta_2(x^*)$ for all $x
\,\in\,\mathcal{A}(\mathbb{C}\mathrm{P}_q^1)$. Connes' bimodule $\Omega^1_{\mathcal{D}}$ of one-forms is identified with the bimodule $M_{-2}\oplus M_2$ (see discussion after Proposition 3.2 in \cite{KhalkhaliLandiVanSuijlekom2011}, and also \cite[Remark 3.6]{Mesland-Rennie}). Thus, the calculus on $\mathcal{A}(\mathbb{C}\mathrm{P}_q^1)$ is also realized via this Dirac operator.

The same phenomenon persists in the setting of $\mathcal{L}_n$'s defined in \eqref{bimod Ln}, where $\mathcal{L}_0\,=\,\mathcal{A}(\mathbb{C}\mathrm{P}_q^1)$, and this is just a different way (left action vs. right action) of looking at the same things, as Lemma \ref{essentially same} shows. In fact, Neshveyev and Tuset \cite{Neshveyev} shows that starting with the differential calculus on $\mathcal{A}(\mathbb{C}\mathrm{P}_q^1)$, one can construct the Dirac operator, and by an important result of \cite{Dabrowski-Sitarz} we know that the Dirac operator on the standard Podle\'s sphere is essentially (up to non-zero scalar multiplication) uniquely determined by the requirements of $SU_q(2)$-invariance, boundedness of commutators, and the first-order condition. In this picture, the closable derivations $\delta_1^\prime(\cdot)\,=\,q^{-\frac{1}{2}}(F\triangleright\,\cdot\,)$ and $\delta_2^\prime(\cdot)\,=\,q^{\frac{1}{2}}(E\triangleright\,\cdot\,)$ are identified with the anti-holomorphic differential $\overline\partial$ and the holomorphic differential $\partial$ on $\mathcal{A}(\mathbb{C}\mathrm{P}_q^1)$ respectively up to the scalar factor $i$. Connes' bimodule $\Omega^1_{\mathcal{D}}$ of one-forms is similarly identified with the bimodule $\mathcal{L}_{-2}\oplus \mathcal{L}_2$. We write $H_\pm$ for the Hilbert space completion of $\mathcal{L}_{\pm 1}$, similar to $\mathcal{H}_\pm$ defined earlier.
\medskip

\noindent\textbf{Notation:} For elements in $\mathcal{A}(\mathbb{C}\mathrm{P}_q^1)$, or more generally in $C(\mathbb{C}\mathrm{P}_q^1)$, henceforth we shall write the $C^*$-norm as $\norm{\cdot}_{C(\mathbb{C}\mathrm{P}_q^1)}$. The norm of $\overline\partial f$ for any $f\in\mathcal{A}(\mathbb{C}\mathrm{P}_q^1)$ will be written as $\norm{\overline\partial f}_{\mathrm{op}}$ viewed as an element in $\mathcal{B}(H_+,\,H_-)$.
\medskip

In this article, we shall concentrate only on the anti-holomorphic differential $\overline\partial$. Denote
\[
\mathcal{B}(H_+,\,H_-):=\{T:H_+\to H_-\,:\,T\mbox{ is a bounded linear operator}\}.
\]
The closable derivation $\overline\partial:\mathcal{A}(\mathbb{C}\mathrm{P}_q^1)\subseteq C(\mathbb{C}\mathrm{P}_q^1)\to \mathcal{B}(H_+,\,H_-)$ extends to a closed derivation (also denoted by the same symbol) $\overline\partial\,:\,\mathscr{D}\, \longrightarrow\, \mathcal{B}(H_+,\,H_-)$ with domain $\mathscr{D}$. For any element $x\in\mathscr{D}$, there exists a sequence $\{x_n\}\subseteq\mathcal{A}(\mathbb{C}\mathrm{P}_q^1)$ such that $x_n\to x$ in $\norm{\cdot}_{C(\mathbb{C}\mathrm{P}_q^1)}$ and $\{\bar\partial(x_n)\}$ converges in the operator norm. Then, $\mathscr{D}$ is a Banach space with the graph norm $\norm{x}_{\mathrm{Gr}}\,:=\,\norm{x}_{C(\mathbb{C}\mathrm{P}_q^1)}+\norm{\overline\partial x}_{\mathrm{op}}$. Since the graph norm dominates the $C^*$-norm $\norm{\cdot}_{C(\mathbb{C}\mathrm{P}_q^1)}$ on $\mathcal{A}(\mathbb{C}\mathrm{P}_q^1)$, it is clear that $\mathscr{D}\,\subseteq\,C(\mathbb{C}\mathrm{P}_q^1)$ is a dense subspace with respect to the $C^*$-norm, as $\mathcal{A}(\mathbb{C}\mathrm{P}_q^1)\,\subseteq\,C(\mathbb{C}\mathrm{P}_q^1)$ is so. Although the following fact may be well-known to experts, we choose to provide a proof for the sake of completeness.

\begin{lemma}\label{Banach algebra}
$\mathscr{D}$ is a unital Banach algebra, and the closure of the closable derivation $\overline\partial:\mathcal{A}(\mathbb{C}\mathrm{P}_q^1)\to \mathcal{B}(H_+,\,H_-)$  is a derivation on $\mathscr{D}$.
\end{lemma}
\begin{proof}
For $a,\,b\,\in\,\mathscr{D}$, there exist sequences $\{a_n\},\,\{b_n\}\,\subseteq\,\mathcal{A}(\mathbb{C}\mathrm{P}_q^1)$ such
that $a_n\,\to\,a$ and $b_n\,\to\,b$ in the $\norm{\cdot}_{\mathrm{Gr}}$-norm. Then, $a_nb_n\,\to \,ab$ in the $C(\mathbb{C}\mathrm{P}_q^1)$-norm. Now, $\overline\partial(a_nb_n)\,=\,a_n\overline\partial(b_n)+\overline\partial(a_n)b_n$. Since convergent sequence of real numbers must be bounded, we have $\norm{\overline\partial b_n}_{\mathrm{op}}\,<\,
M$ for some $M\,>\,0$ and for all $n\,\ge\, 1$. Given $\epsilon\,>\,0$ there exists $N\,\in\,\mathbb{N}$ such
that we have
\[
\norm{a_n\overline\partial(b_n)-a\overline\partial b}_{\mathrm{op}}\,\le\,
\norm{a_n-a}_{C(\mathbb{C}\mathrm{P}_q^1)}\norm{\overline\partial b_n}_{\mathrm{op}}+
\norm{a}_{C(\mathbb{C}\mathrm{P}_q^1)}\norm{\overline\partial(b_n-b)}_{\mathrm{op}}\,<\,\epsilon
\]
for all $n\,\ge \,N$. This proves that $a_n\overline\partial(b_n)\,\to\, a\overline\partial b$, and similarly
$\overline\partial(a_n)b_n\,\to\, \overline\partial a\,.\,b$. So $\overline\partial(a_nb_n)\,\to\,
a\overline\partial b+\overline\partial a.b$. By the closedness of $\overline\partial$, we have $ab\,\in\,
\mathscr{D}$ and $\overline\partial(ab)
\,=\,a\overline\partial b+\overline\partial a\,.\,b$. Thus, $\mathscr{D}$ is an algebra and $\overline\partial$ is still a derivation on $\mathscr{D}$.

Finally, using the derivation property of $\overline\partial$, we obtain
\begin{eqnarray*}
\norm{ab}_{\mathrm{Gr}} &=& \norm{ab}_{C(\mathbb{C}\mathrm{P}_q^1)}+\norm{\overline\partial(ab)}_{\mathrm{op}}\\
&\le& \norm{a}_{C(\mathbb{C}\mathrm{P}_q^1)}\norm{b}_{C(\mathbb{C}\mathrm{P}_q^1)}+\norm{a}_{C(\mathbb{C}\mathrm{P}_q^1)}\norm{\overline\partial b}_{\mathrm{op}}+\norm{b}_{C(\mathbb{C}\mathrm{P}_q^1)}\norm{\overline\partial a}_{\mathrm{op}}\\
&\le& \norm{a}_{C(\mathbb{C}\mathrm{P}_q^1)}\norm{b}_{C(\mathbb{C}\mathrm{P}_q^1)}+\norm{a}_{C(\mathbb{C}\mathrm{P}_q^1)}\norm{\overline\partial b}_{\mathrm{op}}+\norm{b}_{C(\mathbb{C}\mathrm{P}_q^1)}\norm{\overline\partial a}_{\mathrm{op}}+\norm{\overline\partial a}_{\mathrm{op}}\norm{\overline\partial b}_{\mathrm{op}}\\
&=& \norm{a}_{\mathrm{Gr}}\norm{b}_{\mathrm{Gr}},
\end{eqnarray*}
which proves that $\mathscr{D}$ is a Banach algebra.
\end{proof}

It may be noted that $\mathscr{D}$ need not be $\ast$-closed. This is because we have no information on $\partial x$ for $x\,\in\,\mathscr{D}$. However, we do not need it in the context of holomorphic structures. In the setting of $C^*$-algebras and closed $\ast$-derivations, one usually works with an equivalent norm on the domain, see for instance \cite[Section 3.3]{Sakai}.


\subsection{Gauge equivalence of holomorphic structures}

Given $f,\,g,\,h\,\in\,\mathcal{A}(\mathbb{C}\mathrm{P}_q^1)$, consider the mixed gauge equation $\overline\partial g\,=\,g\overline\partial f-\overline\partial h\,.\,g$ in Definition \ref{mixed gauge equation}. Strictly speaking, as an element in $\mathcal{B}(H_+,\,H_-)$, this equation is understood by  $\overline\partial g
\,=\,\pi_-(g)\overline\partial f-\overline\partial h\,\pi_+(g)$ due to \eqref{representation}. Recall that $\pi_\pm$ are not individually faithful, although their direct sum is. If $g$ is an invertible solution of this equation, then $\pi_\pm(g)\,\neq\, 0$ because $1\,=\,\pi_\pm(gg^{-1})\,=\,\pi_\pm(g)\pi_\pm(g^{-1})$. Thus, when we restrict our attention only to the invertible solutions of the equation $\overline
\partial g\,=\,\pi_-(g)\overline\partial f-\overline\partial h\,\pi_+(g)$, non-faithfulness of $\pi_\pm$ does not create any trouble. So we can safely omit the representation symbols $\pi_\pm$.

We have already discussed at the end of Section 3.1 that we need to enlarge the domain $\mathcal{A}(\mathbb{C}\mathrm{P}_q^1)$ to find invertible solutions of the gauge/mixed gauge equation. Using the Dirac operator prescription in Section 3.2, combining Lemma \ref{their gauge}, Def. \ref{mixed gauge equation}, and Lemma \ref{Banach algebra}, we finally propose the following.

\begin{definition}\label{main definition}
Two (flat) $\overline\partial$-connections $\overline{\nabla}_j^{(n)}\,:\,\mathcal{L}_n\,\longrightarrow\,
\Omega^{0,1}(\mathbb{C}\mathrm{P}_q^1)\otimes_{\mathcal{A}(\mathbb{C}\mathrm{P}_q^1)} \mathcal{L}_n$, for $j\,=\,1,\,2$, on the line bundle $\mathcal{L}_n$ over the quantum projective line, given by
\[
\overline{\nabla}_1^{(n)}(\cdot)=\overline{\nabla}_0^{(n)}(\cdot)-\Phi_{(n)}(\cdot\otimes\bar\partial f_1)\,,\quad
\overline{\nabla}_2^{(n)}(\cdot)=\overline{\nabla}_0^{(n)}(\cdot)-\Phi_{(n)}(\cdot\otimes\bar\partial f_2)
\]
(see \eqref{all connection}) with $f_1,f_2\in\mathcal{A}(\mathbb{C}\mathrm{P}_q^1)$, are called \emph{gauge equivalent} if the mixed gauge equation $\bar\partial g=g\bar\partial f_1-\bar\partial f_2\,.\,g\,$ has an invertible solution in $\mathscr{D}$.
\end{definition}

Note that when we say that $g\,\in\,\mathscr{D}$ is invertible, it is implicit that the inverse $g^{-1}$ also lies in $\mathscr{D}$ and not just in $C(\mathbb{C}\mathrm{P}_q^1)$, that is, the Banach algebra structure on $\mathscr{D}$ plays a crucial role here. Thus, our definition \ref{main definition} is not strikingly different, but rather natural compared to Definition $3.2$ proposed in \cite{GravemanLaRueMacArthurPesinWei2025NonStandardHolomorphic}.

Note that if we have an invertible element $g\,\in\,\mathscr{D}$ that solves the gauge equation $\overline\partial g\,=\,g\overline\partial f$, then we have $\overline\partial g^{-1}\,=\,-\overline\partial f\,.\,g^{-1}$. Moreover, the
converse is also true, that is, the equality $\overline\partial g^{-1}\,=\,-\overline\partial f\,.\,g^{-1}$ implies $\overline\partial g\,=\,g\overline\partial f$.
\smallskip

\noindent\textbf{Notation:} We write $\overline{\nabla}^{(n)}_{\overline\partial f}\ \equiv_g\
\overline{\nabla}^{(n)}_{\overline\partial h}$ to mean the connections on $\mathcal{L}_n$ are gauge
equivalent, and $\overline{\nabla}^{(n)}_{\overline\partial f}\ \not\equiv_g\
\overline{\nabla}^{(n)}_{\overline\partial h}$ means that they are not so.
\smallskip

This seems to be the proper place to highlight the complicated nature of the gauge equation. Recall the notion of the defective spot from \cite[Def. 3.25]{GravemanLaRueMacArthurPesinWei2025NonStandardHolomorphic}. For $f\in C^*\{1,B_0\}\subseteq C(\mathbb{C}\mathrm{P}_q^1)$, the commutative unital $C^*$-subalgebra of $C(\mathbb{C}\mathrm{P}_q^1)$ generated by the self-adjoint generator $B_0$, we say that $f$ has a defective spot at $n\in\mathbb{N}$ if
\[
\Psi(f)(q^{2n})-\Psi(f)(q^{2n-2})=1
\]
where $\Psi:C^*\{1,B_0\}\to C(\sigma(B_0))$ is the continuous functional calculus isomorphism.

\begin{lemma}\label{complicacy}
There exist at least contably many non-constant $f\in\mathrm{Pol}(B_0)\subseteq\mathcal{A}(\mathbb{C}\mathrm{P}_q^1)$ for which non-zero scalar $\alpha\in\mathbb{C}^\times$ exists such that $\overline{\nabla}^{(n)}_{\overline\partial f}\not\equiv_g\overline{\nabla}^{(n)}_{\overline\partial (\alpha f)}$.
\end{lemma}
\begin{proof}
Let $m\in\mathbb{N}\setminus\{0\}$ and consider $f_m:=(q^{2m}-q^{2m-2})^{-1}\,B_0$. Then, each $f_m$ has a defective spot at $m\in\mathbb{N}\setminus\{0\}$. By \cite[Theorem 3.27]{GravemanLaRueMacArthurPesinWei2025NonStandardHolomorphic}, no invertible solution of the gauge equation $\bar\partial g=g\bar\partial f_m$ exists for any $f_m$. Hence, $\overline{\nabla}^{(n)}_{\overline\partial f_m}\not\equiv_g\overline{\nabla}^{(n)}_0$, where $\overline{\nabla}^{(n)}_0$ is the standard holomorphic structure on the line bundle $\mathcal{L}_n$. Now, consider $g_m:=qf_m$. Observe that there is no defective spot for any $g_m$ as $0<q<1$. Hence, we get $\overline{\nabla}^{(n)}_{\overline\partial g_m}\equiv_g\overline{\nabla}^{(n)}_0$ by the same result of \cite{GravemanLaRueMacArthurPesinWei2025NonStandardHolomorphic}. Consequently, $\overline{\nabla}^{(n)}_{\overline\partial f_m}\not\equiv_g\overline{\nabla}^{(n)}_{\overline\partial g_m}$ for any $m\in\mathbb{N}\setminus\{0\}$.
\end{proof}

Thus, the holomorphic structures corresponding to any given $f\in\mathcal{A}(\mathbb{C}\mathrm{P}_q^1)$ and $\alpha f$, for $\alpha\in\mathbb{C}^\times$, need not be the same up to gauge equivalence. The main result of 
\cite{GravemanLaRueMacArthurPesinWei2025NonStandardHolomorphic} is that for the trivial line bundle $\mathcal{L}_0$, there exists a countable family 
$$\{f_j\,\,\big\vert\,\,\, f_j\,\in\,\mathrm{Pol}(B_0)\subseteq \mathcal{A}(\mathbb{C}\mathrm{P}_q^1)\,,\,j\,\in\,\mathbb{N}\}$$ such that 
$\overline{\nabla}^{(0)}_{\overline\partial f}\,\not\equiv_g\,\overline{\nabla}^{(0)}_0\,=\,\overline\partial$ for each $j$, 
and $\overline{\nabla}^{(0)}_{\overline\partial f_i}\,\not\equiv_g\,\overline{\nabla}^{(0)}_{\overline\partial f_j}$ for 
$i\neq j$. Using the one-to-one correspondence \cite[Proposition 3.7]{GravemanLaRueMacArthurPesinWei2025NonStandardHolomorphic}) between sets of gauge equivalence classes of holomorphic structures on $\mathcal{L}_n\mbox{ and }\mathcal{L}_m$ for any $n\mbox{ and }m$, the authors finally conclude that on each line bundle $\mathcal{L}_n$, there are at least countably many distinct holomorphic structures. A crucial ingredient in their analysis was that invertible and nonzero non-invertible solutions of the gauge equation cannot co-exist.

\begin{lemma}\cite[Lemma 3.6]{GravemanLaRueMacArthurPesinWei2025NonStandardHolomorphic}\label{their}
If there is a nonzero non-invertible solution of $\overline\partial g\,=\,g\overline\partial f$ in
$\mathscr{D}$, then there does not exist any invertible solution.
\end{lemma}

\begin{corollary}
The space of invertible solutions of $\overline\partial g\,=\,g\overline\partial f$ is at most one-dimensional.
\end{corollary}

\begin{proof}
Let $g_1,\,g_2$ be two invertible solutions, if exist, of the gauge equation $\overline\partial g\,=\,g\overline\partial f$. Take
$a\,=\,g_1g_2^{-1}\in\mathscr{D}$ and obtain $g_1\overline\partial f=\overline\partial g_1\,=\,\overline\partial(ag_2)\,=\,
\overline\partial a\,.\,g_2+ag_2\overline\partial f$. This gives $\overline\partial a\,=\,0$, and consequently,
$g_1\,=\,\lambda g_2$ for some $\lambda\,\in\,\mathbb{C}^\times$, as $\mathrm{ker}(\bar\partial)=\mathbb{C}$.
\end{proof}

However, the situation for the mixed gauge equation $\overline\partial g\,=\,g\overline\partial f-\overline\partial h\,.\,g$ is more tricky. In this case, invertible and nonzero non-invertible solutions can co-exist unless the following condition is met.

\begin{lemma}\label{ours}
Let $f,\, h\,\in\,\mathcal{A}(\mathbb{C}\mathrm{P}^1_q)\setminus\mathbb{C}$. Assume that every $a\,\in\,\mathscr{D}$ of the form $ug$, where $u\,\in\,\mathscr{D}$ is nonzero non-invertible and $g\,\in\,\mathscr{D}$ is invertible, satisfying $\overline\partial a
\,=\,a\overline\partial h-\overline\partial h\,.\,a$ is a scalar multiple of the identity. If there exists a nonzero non-invertible solution $u\,\in\,\mathscr{D}$ of the mixed gauge equation
\[
\overline\partial g\ =\ g\,\overline\partial f \;-\; \overline\partial h\cdot g\,,
\]
then there is no invertible solution of the same equation.
\end{lemma}

\begin{proof}
To prove by contradiction, assume there exists a nonzero non-invertible $u\,\in\,\mathscr{D}$ and an invertible $g\,\in\,\mathscr{D}$ both solving the mixed gauge equation $\overline\partial x\,=\,x\,\overline\partial f \;-\; \overline\partial h\cdot x$. Take $a\,:=\, ug^{-1}\,\in\,\mathscr{D}$. Then,
\[
  \overline\partial u
  \;=\;
  \overline\partial(a g)
  \;=\;
  (\overline\partial a)\,g \;+\; a\,\overline\partial g.
\]
On the other hand, we have
\[
  \overline\partial u
  \;=\;
  u\,\overline\partial f \;-\; \overline\partial h\cdot u
  \;=\;
  a g\,\overline\partial f \;-\; \overline\partial h\cdot a g.
\]
Therefore,
\[
  (\overline\partial a)\,g \;+\; a\,(g\,\overline\partial f - \overline\partial h\cdot g)
  \;=\;
  a g\,\overline\partial f \;-\; \overline\partial h\cdot a g.
\]
The terms $a g\,\overline\partial f$ cancel on both sides and we obtain
\[
  (\overline\partial a)\,g \;-\; a\,\overline\partial h\cdot g
  \;=\;
  -\,\overline\partial h\cdot a g.
\]
Since $g\in\mathscr{D}$ is invertible, it follows that
\[
  \overline\partial a
  \;=\;
  a\,\overline\partial h \;-\; \overline\partial h\cdot a.
\]
By assumption, we have $a \,=\, \lambda$ for some $\lambda\,\in\,\C$. If $\lambda \,=\, 0$, then $u\, =\, 0$, contradicting the assumption that $u$ is nonzero. If $\lambda\,\neq\, 0$, then $u$ is invertible because $g$ is invertible, and thus contradicting the assumption that $u$ is non-invertible. In either case, we reach a
contradiction, and hence no such invertible solution $g\,\in\,\mathscr{D}$ can exist.
\end{proof}
\smallskip


\section{Fixed Points and the Gauge equation}

In this section, we formulate the mixed gauge equation in Definition \ref{main definition} as a fixed--point problem on an appropriate Banach space.

\subsection{A bounded left-inverse for $\overline\partial$ modulo scalars}

Recall that $\overline\partial$ is injective modulo scalars, and we need a bounded operator that works as a left inverse of $\overline\partial$ modulo scalars. In the context of compact quantum metric spaces of Rieffel \cite{Rieffel}, Aguilar and Kaad \cite{AguilarKaad2018PodlesSpectralMetric} constructed one such bounded linear operator, which they refer to as the quantum integral, as it enjoy similar properties to the classical Volterra operator. We briefly recall it here.

First, apply the change of notations $a\leftrightarrow a^*$ and $c\leftrightarrow b^*$ for the $SU_q(2)$-generators in \eqref{SUq(2)} to match with \cite[Section 3]{AguilarKaad2018PodlesSpectralMetric} (see also \cite[Section 2]{Kaad-Kyed}). Then the generators for the coordinate algebra $\mathcal{A}(\mathbb{C}\mathrm{P}_q^1)$ of the quantum projective line in \eqref{Sq^2} change by $B_0\longleftrightarrow A,\,q^{-1}B_+\longleftrightarrow B,\,q^{-1}B_-\longleftrightarrow B^*$ to those in \cite[Section 3]{AguilarKaad2018PodlesSpectralMetric}. The relations in \eqref{Sq^21} coincide with those in \cite{AguilarKaad2018PodlesSpectralMetric}. Recall that the spectrum of the self-adjoint generator $B_0$ of the $C^*$-algebra $C(\mathbb{C}\mathrm{P}_q^1)$ is $\{0\}\cup\{q^{2n}:n\in\mathbb{N}\}$. For $k\in\mathbb{N}$, define the subspace $Y_k:=C(\mathbb{C}\mathrm{P}_q^1)\,.\,\chi_{\{q^{2k}\}}(B_0)\subseteq C(\mathbb{C}\mathrm{P}_q^1)$, which is closed in the $C^*$-norm because $\chi_{\{q^{2k}\}}(B_0)\in C(\mathbb{C}\mathrm{P}_q^1)$ is a projection. Then, $Y_k$ is thought to be fiber over the point $q^{2k}\in\sigma(B_0)$ \cite[Section 6]{AguilarKaad2018PodlesSpectralMetric}. Define
\[
X:=\{\xi\in\mathcal{B}(H_+,\,H_-):\xi\,.\,\chi_{\{q^{2k}\}}(B_0)\in Y_k\,.\,c^2\,\,\forall\,\,k\in\mathbb{N}\},
\]
which is a closed subspace of $\mathcal{B}(H_+,\,H_-)$ \cite[Section 7]{AguilarKaad2018PodlesSpectralMetric}. On each fiber $Y_k$, we have the bounded vertical and horizontal shift operators (see \cite[Eqn. 7.1]{AguilarKaad2018PodlesSpectralMetric})
\begin{align*}
S^V &: Y_k\to Y_k\qquad S^V(f_{n,k}):=f_{n+1,k}\\
S^H &: Y_k\to Y_{k+1}\qquad S^H(f_{n,k}):=f_{n+1,k+1}
\end{align*}
where $f_{n,k}$'s, for $n,k\in\mathbb{N}$, are the matrix units (see \cite[Eqn. 4.2]{AguilarKaad2018PodlesSpectralMetric}) of the standard Podle\'s sphere given by
\[
f_{n,k}:=\begin{cases}
C_{n,k}^{-\frac{1}{2}}\,(q^{-1}B_+)^{n-k}\,.\,\chi_{\{q^{2k}\}}(B_0)\,, & n\ge k\,; \\
C_{n,k}^{-\frac{1}{2}}\,(q^{-1}B_-)^{k-n}\,.\,\chi_{\{q^{2k}\}}(B_0)\,, & 0\le n\le k-1\,.
\end{cases}
\]
with certain scalars $C_{n,k}$ defined in \cite[Eqn. 4.1]{AguilarKaad2018PodlesSpectralMetric}. Consider the bounded diagonal operator $\Gamma:Y_k\to Y_k$, for each $k\in\mathbb{N}$, given by
\[
\Gamma(f_{n,k}):=\begin{cases}
q^{n-k-1}(1-q^{2(k+1)})^{\frac{1}{2}}(1-q^{2n})^{-\frac{1}{2}}f_{n,k}\,, & n\ge k+1\,; \\
0\,, & 0\le n\le k\,.
\end{cases}
\]
which has the operator norm equal to one. The vertical quantum integral $\int^V:X\to C(\mathbb{C}\mathrm{P}_q^1)$ is defined by
\[
\int^V\xi:=-\sum_{m=0}^\infty\,\frac{q^m(1-q^2)}{\sqrt{1-q^{2(m+1)}}}\sum_{\ell=0}^m\,\Gamma(S^H\Gamma)^{m-1}S^V\left(\xi\,.\,\chi_{\{q^{2\ell}\}}(B_0)(c^*)^2q^{-2\ell}\right)\,.
\]
Now, for each $k\in\mathbb{N}\setminus\{0\}$, define the bounded diagonal operator $\Delta:Y_k\to Y_k$ by
\[
\Delta(f_{n,k}):=\begin{cases}
q^{k-n+1}(1-q^{2(n-1)})^{\frac{1}{2}}(1-q^{2k})^{-\frac{1}{2}}f_{n,k}\,, & 0< n\le k\,; \\
0\,, & \mbox{otherwise}.
\end{cases}
\]
with operator norm less than equal to $q$ for all $k\in\mathbb{N}$. Furthermore, for each $k\in\mathbb{N}$, consider the orthogonal projection $P:Y_k\to Y_k$ given by
\[
P(f_{n,k}):=\begin{cases}
f_{n,k}\,, & 0 \le n\le k\,; \\
0\,, & n \ge k+1\,.
\end{cases}
\]
The horizontal quantum integral $\int^H:X\to C(\mathbb{C}\mathrm{P}_q^1)$ is defined by
\[
\int^H\xi:=\sum_{m=0}^\infty\,\frac{q^m(1-q^2)}{\sqrt{1-q^{2(m+1)}}}\sum_{\ell=m+1}^\infty\,(S^H)^*((S^H)^*\Delta)^{\ell-m-1}PS^V\left(\xi\,.\,\chi_{\{q^{2\ell}\}}(B_0)(c^*)^2q^{-2\ell}\right)
\]
Finally, the quantum integral
\[
\int\ \,:\ \, X\ \longrightarrow\ C(\mathbb{C}\mathrm{P}_q^1)
\]
is defined by $\int:=\int^V+\int^H$ \cite[Deinition 7.11]{AguilarKaad2018PodlesSpectralMetric}, and the operator norm of $\int$ is bounded by $\,2\sqrt{1-q^2}\,(1-q)^{-2}$ (notice the $q$-dependence).

Recall that we have an isomorphism of $C^*$-algebras $C(\mathbb{C}\mathrm{P}_q^1)\,\cong\,\mathcal{K}(\ell^2(\mathbb{N}))\oplus\mathbb{C}$ (minimal unitization of the compact operators). Consider the positive linear functional $\psi_\infty$ on the $C^*$-algebra $C(\mathbb{C}\mathrm{P}_q^1)$ defined by $\psi_\infty(x\oplus\lambda)\,:=\,\lambda$ using the isomorphism. In fact, $\psi_\infty$ is the unique character of $C(\mathbb{C}\mathrm{P}_q^1)$. Although $\psiinf$ may be denoted by the more familiar notation $\varepsilon$, we stick to the notation used in \cite{AguilarKaad2018PodlesSpectralMetric} for the reader's convenience.

We need the following result rather than the explicit construction of the quantum integral, and we shall frequently use it in the sequel without further mention.

\begin{theorem}\cite[Theorem 7.1]{AguilarKaad2018PodlesSpectralMetric}\label{integral}
For any $x\in\mathcal{A}(\mathbb{C}\mathrm{P}_q^1)$, one has
\begin{equation*}
\int \overline{\partial} x\ \, =\ \, x - \psiinf(x)\,.\,1_{\mathcal{A}(\mathbb{C}\mathrm{P}_q^1)}\,.
\end{equation*}
\end{theorem}

In fact, the above result holds beyond $\mathcal{A}(\mathbb{C}\mathrm{P}_q^1)$ (for any elements in the Lipschitz algebra corresponding to the D\polhk{a}browski--Sitarz spectral triple), but at the moment we only need this much. For our purpose, in fact we need the domain of the quantum integral restricted to $\overline{\overline\partial(\mathcal{A}(\mathbb{C}\mathrm{P}_q^1))}^{\norm{\cdot}_{\mathrm{op}}}\,\subseteq\, X$ only, the closure in $\mathcal{B}(H_+,\,H_-)$ of the image of $\mathcal{A}(\mathbb{C}\mathrm{P}_q^1)$ under $\bar\partial$, as the following result demonstrates. Recall the Banach algebra $\mathscr{D}$ from Lemma \ref{Banach algebra}.

\begin{proposition}\label{main 1}
For any $f\,\in\,\mathcal{A}(\mathbb{C}\mathrm{P}_q^1)$ and $g\,\in\,\mathscr{D}$,
\[
f\overline\partial g\,,\,\overline{\partial} f \cdot g\,\ \in\,\ X\,\ =\,  \mathrm{Dom}(\int)
\]
and
\[
\psiinf \left( \int f\overline{\partial} g \right)\,\ =\,\ \psiinf \left( \int \overline{\partial} f \cdot g \right)
\,\ =\,\ 0.
\]
As a consequence, the following two hold: $\overline\partial(fg)\,\in\, X$ and $\int\overline\partial(fg)
\,=\,fg-\psiinf(fg)$.
\end{proposition}

\begin{proof}
Pick $(g,\, \overline{\partial}g)\,\in\, \mathrm{Graph}(\overline\partial)$ such that $g\,\in\,\mathscr{D}$. There exists a sequence $(g_n,\, \overline{\partial} g_n) \,\in\,\mathrm{Graph}(\overline{\partial})$ such that $\{g_n\}\, \subseteq\,\mathcal{A}(\mathbb{C}\mathrm{P}_q^1)$ with
$g_n \,\to\, g$ in the $C(\mathbb{C}\mathrm{P}_q^1)$-norm,
and $\overline{\partial} g_n \,\to\, \overline{\partial} g$ in the operator norm. Since multiplication is
continuous with respect to the operator norm, we have $f \overline{\partial} g_n \,\to\, f \overline{\partial} g$
for any $f \,\in\, \mathcal{A}(\mathbb{C}\mathrm{P}_q^1)$. Since $\overline\partial\,:\,\mathcal{A}(\mathbb{C}\mathrm{P}_q^1)\,
\longrightarrow\, \Omega^{0,1}$ is surjective and $\Omega^{0,1}$ is a bimodule over $\mathcal{A}(\mathbb{C}\mathrm{P}_q^1)$,
we have $f \overline{\partial} g_n \,\in\, X$ because $\overline\partial(\mathcal{A}(\mathbb{C}\mathrm{P}_q^1))
\,\subseteq\, X$. Since $X \,\subseteq\, \mathcal{B}(H_+,\,H_-)$ is closed, we obtain $f \overline{\partial} g \,\in\, X$.

Now, for the second part, we have
\[
\int f \overline{\partial} g\ =\ \int \lim_n f \overline{\partial} g_n 
\ =\ \lim_n \int f \overline{\partial} g_n,
\]
because $\int$ is a bounded operator. We again use the surjectivity of $\overline\partial\,:\,\mathcal{A}(\mathbb{C}\mathrm{P}_q^1)\,
\longrightarrow\, \Omega^{0,1}$.
For each $n$, we have $f \overline{\partial} g_n\, =\, \overline{\partial} b_n$ for some $b_n
\,\in\,\mathcal{A}(\mathbb{C}\mathrm{P}_q^1)$. Therefore,
\begin{align*}
\psiinf \left( \int f \overline{\partial} g \right) 
  &= \psiinf \Bigl( \lim_n \int f\overline{\partial} g_n \Bigr) \\
  &= \psiinf \Bigl( \lim_n \int \overline{\partial} b_n \Bigr)\\
  &= \psiinf \Bigl( \lim_n (b_n - \psiinf(b_n) \cdot 1) \Bigr) \\
  &= \lim_n \psiinf\bigl(b_n - \psiinf(b_n) \cdot 1\bigr) \\
  &= 0\,,
\end{align*}
as $\psiinf$ is continuous with respect to $\norm{\cdot}_{C(\mathbb{C}\mathrm{P}_q^1)}$.

To prove the statement for $\overline\partial f\,.\,g$, start with the same sequence $\{g_n\}\,\subseteq\,\mathcal{A}(\mathbb{C}\mathrm{P}_q^1)$ such that $g_n\,\to\, g$ in the $\norm{.}_{\mathrm{gr}}$-norm on $\mathscr{D}$. Then, $\overline\partial f\,.\,g_n\,\to\,\overline\partial f\,.\,g$ in the operator norm as the multiplication
is continuous. Using the same reasoning as above, $\overline\partial f\,.\,g\,\in\, X$ and $\int\overline\partial f\,.\,g_n\,\to\,\int\overline\partial f\,.\,g$. We have $\psiinf(\int \overline\partial f.g)\,=\,\lim\psiinf(\int \overline\partial f.g_n)\,=\,0$.

For the final part, we have that $\mathscr{D}$ is an algebra and the Leibniz rule
\[
\overline{\partial}(fg)\ =\ \overline{\partial} f \cdot g + f \overline{\partial} g
\]
holds for $f,\,g\,\in\,\mathscr{D}$. Since $X$ is a closed subspace of
$\mathcal{B}(H_+,\,H_-)$, it follows that $\overline{\partial}(fg)
\,\in \,X$. Finally, using the linearity and continuity of $\int$ the following is deduced:
\begin{eqnarray*}
\int \overline{\partial}(fg)=\int\overline\partial f\,.\,g+\int f\overline\partial g &=& \lim\left(\int\overline\partial f\,.\,g_n+f\overline\partial g_n\right)\\
&=& \lim\int\overline\partial(fg_n)\\
&=& \lim(fg_n-\psiinf(fg_n))\\
&=& fg-\psiinf(fg),
\end{eqnarray*}
which completes the proof.
\end{proof}

\begin{remark}
It is not known {\em a priori} whether $\overline\partial(fg)\,\in\, X$ for $f\,\in\,\mathcal{A}(\mathbb{C}\mathrm{P}_q^1)$ and
$g\,\in\,\mathscr{D}$. This is because the quantum integral is designed to work for elements in the Lipschitz algebra associated with the D\polhk{a}browski-Sitarz spectral triple. For $g\,\in\,\mathscr{D}$, we have that
$\overline\partial g$ is a well-defined bounded operator; however, we do not know whether $\partial g$ is also so. Thus, $fg\,\in\,\mathscr{D}$ need not lie in the Lipschitz algebra {\em a priori}, so we cannot write $\int\overline\partial(fg)$. The above proof is designed accordingly to overcome this difficulty.
\end{remark}

\begin{corollary}\label{001}
For any $g\in\mathscr{D},\,\bar\partial g\in X$ and $$\int\overline\partial g\ =\ g-\psiinf(g).$$
\end{corollary}

\begin{lemma}\label{main 2}
For any $f\,\in\,\mathcal{A}(\mathbb{C}\mathrm{P}_q^1)$ and $g \,\in\,\mathscr{D}$,
$$\int f\overline\partial g\,,\,\int\overline{\partial} f\,.\,g\,\in\,\mathscr{D}.$$
Moreover, $(\overline\partial\circ\int)( f\overline\partial g)\,=\, f\overline\partial g$ and
$(\overline\partial\circ\int)( \overline\partial f\,.\,g)\,=\,\overline\partial f\,.\,g$.
\end{lemma}

\begin{proof}
For $g\in\mathscr{D}$, consider $\{g_n\}\,\subseteq\,\mathcal{A}(\mathbb{C}\mathrm{P}_q^1)$ such that $g_n\,\to\, g$ in the $C^*$-norm of $C(\mathbb{C}\mathrm{P}_q^1)$ and 
$\overline\partial g_n\,\to\,\overline\partial g$ in the $\norm{.}_{\mathrm{op}}$-norm. Then, $f\overline\partial 
g_n\,\to\, f\overline\partial g$ in the $\norm{.}_{\mathrm{op}}$-norm for any $f\,\in\,\mathcal{A}(\mathbb{C}\mathrm{P}_q^1)$. 
Hence, $\int f\overline\partial g_n\,\to\,\int f\overline\partial g$ in $\norm{\cdot}_{C(\mathbb{C}\mathrm{P}_q^1)}$. Since 
$\overline\partial\,:\,\mathcal{A}(\mathbb{C}\mathrm{P}_q^1)\,\longrightarrow\,\Omega^{0,1}$ is surjective, for each $n$ we have 
$f\overline\partial g_n\,=\,\overline\partial y_n$ for $y_n\,\in\,\mathcal{A}(\mathbb{C}\mathrm{P}_q^1)$, and hence $$\int 
f\overline\partial g_n\ =\ \int\overline\partial y_n\ =\ y_n-\psiinf(y_n)\ \in\ \mathcal{A}(\mathbb{C}\mathrm{P}_q^1).$$ Moreover, 
$(\overline\partial\circ\int)(f\overline\partial g_n)\,=\,\overline\partial y_n\,=\,f\overline\partial g_n$.
Consequently, we have 
$\int f\overline\partial g_n\,\in\,\mathcal{A}(\mathbb{C}\mathrm{P}_q^1)$, and $\int f\overline\partial g_n\,\to\, 
\int f\overline\partial g$ in $C(\mathbb{C}\mathrm{P}_q^1)$-norm with $$\overline\partial\circ\int (f\overline\partial 
g_n)\ =\ f\overline\partial g_n\ \to\ f\overline\partial g$$ in $\norm{\cdot}_{\mathrm{op}}$. Using the closedness of 
$\overline\partial$ it follows that $\int f\overline\partial g\,\in\,\mathscr{D}$ and $(\overline\partial\circ\int)(f\overline\partial g)\,=\,f\overline\partial g$.

The other part follows by the Leibniz rule and Corollary \ref{001}.
\end{proof}

It is now clear that for our purpose, in the context of the mixed gauge equation, we need the quantum integral restricted to 
$\overline{\overline\partial(\mathcal{A}(\mathbb{C}\mathrm{P}_q^1))}^{\norm{\cdot}_{\mathrm{op}}}\ \subseteq\ X$ only.

\noindent\textbf{Notation:} We denote the operator norm of the quantum integral restricted to $\overline{\overline\partial(\mathcal{A}(\mathbb{C}\mathrm{P}_q^1))}\subseteq X$ by $M$ throughout the article. Hence, $M\ \le\ 2\sqrt{1-q^2}\,(1-q)^{-2}$.

\begin{lemma}
If $x\,\in\, X$ satisfies conditions $\int x\,\in\,\mathscr{D}$ and $\overline\partial\circ\int x\,=\,x$, then
\begin{enumerate}
\item[$(i)$] $\norm{\int x}_{C(\mathbb{C}\mathrm{P}_q^1)}\, \le\, M\norm{x}_{\mathrm{op}}$, and
\item[$(ii)$] $\norm{\int x}_{\mathrm{Gr}}\,\le\, (M+1)\norm{x}_{\mathrm{op}}$.
\end{enumerate}
\end{lemma}

\begin{proof}
We have the following:
\begin{center}
$\norm{\int x}_{\mathrm{Gr}}=\norm{\int x}_{C(\mathbb{C}\mathrm{P}_q^1)}+\norm{\overline\partial(\int x)}_{\mathrm{op}}\le M\norm{x}_{\mathrm{op}}+\norm{x}_{\mathrm{op}}=(M+1)\norm{x}_{\mathrm{op}}.$
\end{center}
\end{proof}

\begin{lemma}
For any $u \,\in\, \mathscr{D}$,
\[
\norm{u}_{C(\mathbb{C}\mathrm{P}_q^1)}\ \le\ M \norm{\overline\partial u}_{\mathrm{op}} + |\psiinf(u)|\,,
\]
and
\[
\norm{u}_{\mathrm{Gr}}\ \le\ (M+1)\norm{\overline\partial u}_{\mathrm{op}}+ |\psiinf(u)|\,.
\]
\end{lemma}

\begin{proof}
Since $u \,\in\, \mathscr{D}$, it follows that $\int \overline{\partial} u\, =\,
u - \psiinf(u)$ by Corollary \ref{001}. Thus,
\[
  \norm{u}_{C(\mathbb{C}\mathrm{P}_q^1)}\ \le\ \norm{\int \overline{\partial} u}_{C(\mathbb{C}\mathrm{P}_q^1)}
  + \norm{\psiinf(u) \cdot 1}_{C(\mathbb{C}\mathrm{P}_q^1)}\ \le\ M \norm{\overline{\partial} u}_{\mathrm{op}} + |\psiinf(u)|.
\]
Moreover,
\[
\norm{u}_{\mathrm{Gr}}\ =\ \norm{u}_{C(\mathbb{C}\mathrm{P}_q^1)}+\norm{\overline\partial(u)}_{\mathrm{op}}\
\le\ (M+1)\norm{\overline\partial(u)}_{\mathrm{op}}+ |\psiinf(u)|.
\]
\end{proof}

\begin{lemma}\label{Banach subspace}
Define $H\ :=\ \mathscr{D} \cap \ker \psiinf$. Then $H$ is a Banach subspace of
$(\mathscr{D}, \norm{\cdot}_{\mathrm{Gr}})$.
\end{lemma}

\begin{proof}
Let $\{\xi_n\}\, \subseteq\, H$, and suppose that $\xi_n \,\to\, \xi$ in $\mathscr{D}$ with respect to the graph norm. Then, $\xi_n\, \to\, \xi$ in the $C(\mathbb{C}\mathrm{P}_q^1)$-norm. Since $\psiinf$ is a continuous linear functional on $C(\mathbb{C}\mathrm{P}_q^1)$, it follows that $\psiinf(\xi_n) \,\to\, \psiinf(\xi)$. But $\psiinf(\xi_n)\, =\, 0$ for all $n$, so
$\psiinf(\xi)\, =\, 0$, i.e.,\ $\xi \,\in\, H$.
\end{proof}

Notice that no element in $H$ can be invertible as $\psiinf$ is a character of $C(\mathbb{C}\mathrm{P}_q^1)$.


\subsection{Fixed-point formulation for $\overline{\partial}u = -\overline{\partial}f\cdot u$}

Recall that $\overline{\partial}u =u\overline{\partial}f$ is identical to $\overline{\partial}u^{-1}=-\overline{\partial}f\cdot u^{-1}$ if $u$ is ivertible. We relate solutions to the gauge equation $\overline{\partial}u = -\overline{\partial}f\,.\,u$ with fixed points of maps in a Banach space. Recall that $\mathscr{D}$ is a Banach space equipped with the graph norm and $H$ is a Banach subspace of it (Lemma \ref{Banach subspace}).

For $f\in\mathcal{A}(\mathbb{C}\mathrm{P}_q^1)$ and $u\in\mathscr{D}$, by Proposition \ref{main 1} and Lemma \ref{main 2} the following linear map
\begin{align*}
\phi:\mathscr{D} &\longrightarrow\mathscr{D}\\
u &\longmapsto -\int\overline\partial f\,.\,u
\end{align*}
is well-defined with $\mathrm{Ran}(\phi)\subseteq H$. Moreover, by Lemma \ref{main 2} we have fixed points of $\phi$ giving solutions of the gauge equation $\overline\partial u=-\overline\partial f\,.\,u$, but not conversely.

We first show that under certain condition on $f\in\mathcal{A}(\mathbb{C}\mathrm{P}_q^1)$, there does not exist any non-zero solution of the gauge equation in the Banach subspace $H$.

\begin{lemma}\label{no solution}
For any $f\in\mathcal{A}(\mathbb{C}\mathrm{P}_q^1)$ with $\norm{\overline\partial f}_{\mathrm{op}}<\frac{1}{M}$, the restriction of the map $\phi$ to the Banach subspace $H$ of $\mathscr{D}$ is a contraction map. Consequently, the gauge equation $\overline\partial u=-\overline\partial f\,.\,u$ admits no nonzero solution in $H$.
\end{lemma}
\begin{proof}
For $u\in H$,
\begin{align*}
\norm{\phi(u)}_{\mathrm{Gr}}=\norm{\int \overline{\partial} f \cdot u}_{\mathrm{Gr}} 
&= \norm{\int \overline{\partial} f \cdot u}_{C(\mathbb{C}\mathrm{P}_q^1)}+\norm{\overline{\partial} \left( \int \overline{\partial} f \cdot u \right)}_{\mathrm{op}}\\
&\le M \norm{\overline{\partial} f \cdot u}_{\mathrm{op}}+\norm{\overline{\partial} f \cdot u}_{\mathrm{op}}\\
&\le M \norm{\overline{\partial} f}_{\mathrm{op}}\norm{u}_{C(\mathbb{C}\mathrm{P}_q^1)}+\norm{\overline{\partial} f}_{\mathrm{op}} \norm{u}_{C(\mathbb{C}\mathrm{P}_q^1)}.
\end{align*}
Now $\int\overline\partial u=u$ since $u\in H$. This implies that $\norm{u}_{C(\mathbb{C}\mathrm{P}_q^1)}\le M\norm{\overline\partial u}_{\mathrm{op}}$. Therefore,
\begin{align*}
\norm{\phi(u)}_{\mathrm{Gr}}  &\le M \norm{\overline{\partial} f}_{\mathrm{op}} \norm{u}_{C(\mathbb{C}\mathrm{P}_q^1)}+M\norm{\overline{\partial} f}_{\mathrm{op}}\norm{\overline\partial u}_{\mathrm{op}}\\
&\le M \norm{\overline{\partial} f}_{\mathrm{op}}\norm{u}_{\mathrm{Gr}}\\
&< \norm{u}_{\mathrm{Gr}}
\end{align*}
by our hypothesis on $f$. Since $\phi$ is a contraction, it has a unique fixed point by the Banach fixed point theorem. Because $\phi$ is linear, $0$ is always a fixed point, and this proves that no nonzero solution to the gauge equation is possible in $H$ for such an $f$.
\end{proof}

Recall that if any nonzero solution of the gauge equation exists in $H$, it cannot be invertible. Consequently, we can immediately conclude that the gauge equation has no invertible solution in $\mathscr{D}$ due to Lemma \ref{their}.


\subsection{Gauge equivalence with the standard holomorphic structure}

We further exploit the fixed--point formulation, break the linearity of $\phi$ considered in the previous subsection, and show that an invertible solution of the gauge equation exists in $\mathscr{D}$ under certain condition on $f$.
\begin{lemma}
For any $f \in \mathcal{A}(\mathbb{C}\mathrm{P}_q^1)$ satisfying $\norm{\overline{\partial} f}_{\mathrm{op}}<\frac{1}{2M}$, the gauge equation $\overline\partial g=-\overline\partial f\,.\,g$ has a solution $g_0\in\mathscr{D}$ such that $g_0$ is an invertible element when viewed in the $C^*$-algebra $C(\mathbb{C}\mathrm{P}_q^1)$.
\end{lemma}
\begin{proof}
The case of $f$ being a scalar is trivial, and hence we assume that $f$ is non-scalar. Consider the following equation
\[
  \overline{\partial} v = \overline{\partial} f(1-v)
\]
in $v \in \mathscr{D}$ for given $f\in\mathcal{A}(\mathbb{C}\mathrm{P}_q^1)$. If $1\neq v_0\in\mathscr{D}$ is a solution with $\norm{v_0}_{C(\mathbb{C}\mathrm{P}_q^1)} < 1$, then $g_0 := 1 - v_0$ is an invertible element in $\mathscr{D}$ with $g_0^{-1}\in C(\mathbb{C}\mathrm{P}_q^1)$, and $g_0$ satisfies
\[
  \overline{\partial} g_0 = -\overline{\partial} f \cdot g_0
\]
because
\[
  \overline{\partial} g_0 = -\overline{\partial} v_0 
  = -\overline{\partial} f(1-v_0) = -\overline{\partial} f \cdot g_0.
\]
Thus, it is enough to find a solution $v_0\in\mathscr{D}$ with $\norm{v_0}_{C(\mathbb{C}\mathrm{P}_q^1)}<1$ for the equation
\[
  \overline{\partial} v = \overline{\partial} f(1-v).
\]
Motivated by the map $\phi$ considered earlier, we now consider $\Phi: H \to H$ defined by
\[
  \Phi(v) := \int \overline{\partial} f(1-v).
\]
That $\Phi$ is well-defined is clear by Proposition \ref{main 1} and Lemma \ref{main 2}. Note that $\Phi$ is not linear since $f$ is non-scalar.

Now for $v_1,v_2\in H$, we have
\begin{align*}
d_{\mathrm{Gr}}(\Phi(v_1),\Phi(v_2)) &=\norm{\Phi(v_1) - \Phi(v_2)}_{\mathrm{Gr}}\\
  &= \norm{\int \overline{\partial} f (1-v_1) - \int \overline{\partial} f (1-v_2)}_{\mathrm{Gr}} \\
  &= \norm{\int \overline{\partial} f (v_2 - v_1)}_{\mathrm{Gr}} \\
  &< \norm{v_2 - v_1}_{\mathrm{Gr}},
\end{align*}
where the last line follows along the same line as in the proof of Lemma \ref{no solution} using our hypothesis on $f$. Thus, $\Phi$ is a contraction map, and hence by the Banach fixed point theorem, there is a unique fixed point $v_0\in H$ of $\Phi$, i.e., we have $v_0=\Phi(v_0)$. Therefore, $\overline{\partial} v_0 = \overline\partial\circ\int\left(\overline{\partial} f(1-v_0)\right)=\overline{\partial} f(1-v_0)$ by Lemma \ref{main 2}. Moreover, it is clear that $v_0\neq 1$ by the definition of $H$.

Finally, we show that $\norm{v_0}_{C(\mathbb{C}\mathrm{P}_q^1)}<1$. We have
\[
  v_0 = \Phi(v_0) = \int \overline{\partial} f(1-v_0).
\]
Therefore,
\[
\norm{v_0}_{C(\mathbb{C}\mathrm{P}_q^1)} \le M \norm{\overline{\partial} f}_{\mathrm{op}} (1 + \norm{v_0}_{C(\mathbb{C}\mathrm{P}_q^1)}).
\]
This says that
\[
 \big(1 - M \norm{\overline{\partial} f}_{\mathrm{op}}\big) \norm{v_0}_{C(\mathbb{C}\mathrm{P}_q^1)}\le M \norm{\overline{\partial} f}_{\mathrm{op}}
 \]
and consequently,
\begin{eqnarray}\label{90}
 \norm{v_0}_{C(\mathbb{C}\mathrm{P}_q^1)}\le \frac{M \norm{\overline{\partial} f}_{\mathrm{op}}}{1 - M \norm{\overline{\partial} f}_{\mathrm{op}}}.
\end{eqnarray}
This implies that $\norm{v_0}_{C(\mathbb{C}\mathrm{P}_q^1)} < 1$ because $\norm{\overline{\partial} f}_{\mathrm{op}} < \frac{1}{2M}$ by the assumption.
\end{proof}

Note that $\norm{v_0}_{C(\mathbb{C}\mathrm{P}_q^1)}<1$ implies that $1-v_0$ is an invertible element in $\mathscr{D}\subseteq C(\mathbb{C}\mathrm{P}_q^1)$ with the inverse in the $C^*$-algebra $C(\mathbb{C}\mathrm{P}_q^1)$. If we need $(1-v_0)^{-1}\in\mathscr{D}$, then it suffices to show that $\norm{v_0}_{\mathrm{Gr}}<1$, as $\mathscr{D}$ is a Banach algebra equipped with the graph norm.

Recall the expression of any arbitrary connection $\overline{\nabla}^{(n)}_{\overline{\partial} f}$ on the line bundle $\mathcal{L}_n$ from \eqref{all connection}. The following gives a criterion for many gauge equivalent holomorphic structures.

\begin{proposition}
For any $f \in \mathcal{A}(\mathbb{C}\mathrm{P}_q^1)$ satisfying $\norm{\overline\partial f}_{\mathrm{op}}<\frac{1}{3M}$, the holomorphic structure $\overline{\nabla}^{(n)}_{\overline{\partial} f}$ is gauge equivalent
to the standard holomorphic structure $\overline{\nabla}_0^{(n)}$ on $\mathcal{L}_n$.
\end{proposition}
\begin{proof}
The case of $f$ being a scalar is trivial, and hence we assume that $f$ is non-scalar. The proof of the previous lemma shows that it is enough to prove $\norm{v_0}_{C(\mathbb{C}\mathrm{P}_q^1)}<\frac{1}{2}$ and $\norm{\overline\partial v_0}_{\mathrm{op}}<\frac{1}{2}$, as this would ensure that $\norm{v_0}_{\mathrm{Gr}}<1$, and consequently $g_0^{-1}\in\mathscr{D}$ for $g_0:=1-v_0$. The present stronger hypothesis ensures that $\norm{v_0}_{C(\mathbb{C}\mathrm{P}_q^1)}<\frac{1}{2}$ from \eqref{90}. We only need to establish $\norm{\overline\partial v_0}_{\mathrm{op}}<\frac{1}{2}$. We have the fixed point equation
\[
 v_0 = \Phi(v_0) = \int \overline{\partial} f(1-v_0).
\]
But, then
\begin{eqnarray*}
\norm{\overline\partial v_0}_{\mathrm{op}}=\norm{\int\,\overline\partial f(1-v_0)}_{\mathrm{op}} &\le& M\norm{\overline\partial f}_{\mathrm{op}}(1+\norm{v_0}_{C(\mathbb{C}\mathrm{P}_q^1)})\\
&<& M\times \frac{1}{3M}\times\frac{3}{2}=\frac{1}{2}
\end{eqnarray*}
using our hypothesis on $f$ and the fact that $\norm{v_0}_{C(\mathbb{C}\mathrm{P}_q^1)}<\frac{1}{2}$. Since $\norm{v_0}_{\mathrm{Gr}}<1$, the element $g_0$ is invertible in the Banach algebra $\mathscr{D}$, and consequently, we have an invertible solution in $\mathscr{D}$ for the gauge equation $\overline\partial g=g\,\overline\partial f$. The result now follows from Definition \ref{main definition}.
\end{proof}

The above is an improvement of \cite[Corollary 3.28]{GravemanLaRueMacArthurPesinWei2025NonStandardHolomorphic} which works only for elements $f\in\mathrm{Pol}(B_0)\subseteq\mathcal{A}(\mathbb{C}\mathrm{P}_q^1)$.


\subsection{Two holomorphic structures and the mixed gauge equation}

We now consider two holomorphic structures $\overline{\nabla}^{(n)}_{\overline{\partial} f}$ and
$\overline{\nabla}^{(n)}_{\overline{\partial} h}$ on $\mathcal{L}_n$ (see \eqref{all connection}), where $f,h\in\mathcal{A}(\mathbb{C}\mathrm{P}_q^1)$, and the associated mixed gauge equation $\overline{\partial} g = g \overline{\partial} f - \overline{\partial} h \cdot g$ for $g\in\mathscr{D}$. For any given $f,\,h\,\in\,\mathcal{A}(\mathbb{C}\mathrm{P}_q^1)$ when we ask whether $\overline{\nabla}^{(n)}_{\overline\partial f}$ and $\overline{\nabla}^{(n)}_{\overline\partial h}$ are gauge equivalent or not, it is understood that $h-f$ is not a scalar, as otherwise the connections would coincide, and hence be trivially gauge equivalent.

Given $g\in\mathscr{D}$, consider $v = 1+g\in\mathscr{D}$. Then $\overline{\partial} v = \overline{\partial} g$, and the equation $\overline{\partial} g = g \overline{\partial} f - \overline{\partial} h \cdot g$ transforms to the following equation~:
\begin{eqnarray}\label{lem:mixed-fixed-point}
\overline{\partial} v &=& (v-1) \overline{\partial} f - \overline{\partial} h (v-1)\nonumber\\
&=& v \overline{\partial} f - \overline{\partial} h \cdot v + \overline{\partial}(h-f).
\end{eqnarray}

Define $\Psi: H \to H$ by
\begin{eqnarray}\label{the main map}
v \longmapsto \int \left(v \overline{\partial} f - \overline{\partial} h \cdot v\right) + \int\overline{\partial}(h-f),
\end{eqnarray}
where $H=\mathscr{D}\cap\ker\psi_\infty$. This map is well-defined by Proposition \ref{main 1} and Lemma \ref{main 2}. Note that $\Psi$ is not a linear map unless $f$ and $h$ differ by a scalar. Observe that if $v_0$ is a fixed point of $\Psi$, then $v_0$ is a solution of the equation \eqref{lem:mixed-fixed-point}, in view of Lemma \ref{main 2}.

\begin{lemma}\label{used next}
Let $f, h \in \mathcal{A}(\mathbb{C}\mathrm{P}_q^1)$ be such that
\[
\norm{\overline\partial f}_{\mathrm{op}}+\norm{\overline\partial h}_{\mathrm{op}}<\frac{1}{M}
\]
Then, the mixed gauge equation
\[
  \overline{\partial} g = g \overline{\partial} f - \overline{\partial} h \cdot g
\]
has a solution $g_0\in\mathscr{D}$. Moreover, if $f-h$ is non-scalar, then $g_0$ is also non-scalar.
\end{lemma}
\begin{proof}
First observe that for $v_1, v_2 \in H$, we have
\[
\norm{v_1-v_2}_{C(\mathbb{C}\mathrm{P}_q^1)}=\norm{\int\overline\partial(v_1-v_2)}_{C(\mathbb{C}\mathrm{P}_q^1)}\le M\norm{\overline\partial(v_1-v_2)}_{\mathrm{op}}.
\]
This gives us the following estimate (applying Lemma \ref{main 2}):
\begin{eqnarray*}
& & d_{\mathrm{Gr}}(\Psi(v_1),\Psi(v_2))\\
&=& \norm{\Psi(v_1) - \Psi(v_2)}_{\mathrm{Gr}}\\
&=& \norm{\int (v_1 - v_2)\overline\partial f - \int \overline\partial h(v_1 - v_2)}_{\mathrm{Gr}} \\
&\le& \norm{\int (v_1 - v_2)\overline\partial f}_{\mathrm{Gr}}+\norm{\int \overline\partial h(v_1 - v_2)}_{\mathrm{Gr}}\\
&\le& \norm{\int (v_1 - v_2)\overline\partial f}_{C(\mathbb{C}\mathrm{P}_q^1)}+\norm{\int \overline\partial h(v_1-v_2)}_{C(\mathbb{C}\mathrm{P}_q^1)}+\norm{(v_1 - v_2)\overline\partial f}_{\mathrm{op}}+\norm{\overline\partial h(v_1 - v_2)}_{\mathrm{op}}\\
&\le& M\norm{v_1 - v_2}_{C(\mathbb{C}\mathrm{P}_q^1)}\left(\norm{\overline\partial f}_{\mathrm{op}}+\norm{\overline\partial h}_{\mathrm{op}}\right)+\norm{v_1 - v_2}_{C(\mathbb{C}\mathrm{P}_q^1)}\left(\norm{\overline\partial f}_{\mathrm{op}}+\norm{\overline\partial h}_{\mathrm{op}}\right)\\
&\le& M\norm{v_1 - v_2}_{C(\mathbb{C}\mathrm{P}_q^1)}\left(\norm{\overline\partial f}_{\mathrm{op}}+\norm{\overline\partial h}_{\mathrm{op}}\right)+M\norm{\overline\partial(v_1 - v_2)}_{\mathrm{op}}\left(\norm{\overline\partial f}_{\mathrm{op}}+\norm{\overline\partial h}_{\mathrm{op}}\right)\\
&=& M\left(\norm{\overline\partial f}_{\mathrm{op}}+\norm{\overline\partial h}_{\mathrm{op}}\right)\norm{v_1-v_2}_{\mathrm{Gr}}\\
&<& \norm{v_1 - v_2}_{\mathrm{Gr}}
\end{eqnarray*}
by our assumption on $f\mbox{ and }h$. So $\Psi$ defined in \eqref{the main map} is a contraction map on the Banach space $H$. By the Banach fixed point theorem, there exists a unique fixed point $v_0$ of $\Psi$. Thus, we have \eqref{lem:mixed-fixed-point} for $v=v_0$, and consequently $g_0 := v_0 - 1$ is a solution of the mixed gauge equation $\overline\partial g=g\overline\partial f-\overline{\partial}h\,.\,g$.

Finally, if $f-h$ is non-scalar, then $v_0\neq 0$ as it satisfies \eqref{lem:mixed-fixed-point}. Since $v_0\in H$ it cannot be equal to any nonzero scalar. Consequently, $g_0=v_0-1$ is non-scalar.
\end{proof}

\begin{lemma}
Let $f,h \in \mathcal{A}(\mathbb{C}\mathrm{P}_q^1)$ be such that $f-h$ is non-scalar and
\[
\norm{\overline{\partial} h}_{\mathrm{op}} + \norm{\overline{\partial} f}_{\mathrm{op}}<\frac{1}{M}.
\]
Then,
\[
\norm{h - f - \psiinf(h-f)}_{C(\mathbb{C}\mathrm{P}_q^1)} < 1.
\]
\end{lemma}
\begin{proof}
We have for $f,h\in\mathcal{A}(\mathbb{C}\mathrm{P}_q^1)$
\[
\int \overline{\partial}(h-f) = h-f - \psiinf(h-f).
\]
Therefore,
\begin{align*}
\norm{\int \overline{\partial}(h-f)}_{C(\mathbb{C}\mathrm{P}_q^1)} 
&\le M \norm{\overline{\partial} h - \overline{\partial} f}_{\mathrm{op}} \\
&\le M \left(\norm{\overline{\partial} h}_{\mathrm{op}} + \norm{\overline{\partial} f}_{\mathrm{op}}\right)\\
&< 1.
\end{align*}
\end{proof}

\begin{lemma}
Suppose $f, h \in\mathcal{A}(\mathbb{C}\mathrm{P}_q^1)$ be such that $f-h$ is non-scalar and
\[
\norm{\overline{\partial} h}_{\mathrm{op}} + \norm{\overline{\partial} f}_{\mathrm{op}}
< \frac{1 - \norm{h - f - \psiinf(h-f)}_{C(\mathbb{C}\mathrm{P}_q^1)}}{M}.
\]
Then, we have a solution $g_0\in\mathscr{D}$ to the mixed gauge equation $\overline\partial g=g\overline\partial f-\overline\partial h\,.\,g$ such that $g_0$ is invertible when viewed as an element in the $C^*$-algebra $C(\mathbb{C}\mathrm{P}_q^1)$.
\end{lemma}
\begin{proof}
We have
\[
0<\frac{1 - \norm{h - f - \psiinf(h-f)}_{C(\mathbb{C}\mathrm{P}_q^1)}}{M} < \frac{1}{M},
\]
and hence
\[
\norm{\overline{\partial} h}_{\mathrm{op}} + \norm{\overline{\partial} f}_{\mathrm{op}}<\frac{1}{M}
\]
by our hypothesis. Therefore, by Lemma \ref{used next}, the mixed gauge equation
\[
\overline\partial g = g \overline\partial f - \overline\partial h g
\]
has a non-scalar solution $g_0$ in $\mathscr{D}$. Moreover, $g_0=v_0-1$ with $v_0\in H$ and $v_0$ is a non-scalar solution of the equation
\[
\overline\partial v = v \overline\partial f - \overline\partial h\,.\,v + \overline\partial(h-f)
\]
in \eqref{lem:mixed-fixed-point}. Since $v_0\in H$,
\[
v_0 = \int (v_0 \overline\partial f - \overline\partial h\,.\,v_0) + \int \overline\partial(h-f),
\]
and hence,
\[
\norm{v_0}_{C(\mathbb{C}\mathrm{P}_q^1)} \le M \norm{v_0}_{C(\mathbb{C}\mathrm{P}_q^1)}\left(\norm{\overline\partial f}_{\mathrm{op}}+\norm{\overline\partial h}_{\mathrm{op}}\right)+\norm{\int \overline\partial(h-f)}_{C(\mathbb{C}\mathrm{P}_q^1)}.
\]
Therefore,
\[
0<\norm{v_0}_{C(\mathbb{C}\mathrm{P}_q^1)}\left( 1 - M\left(\norm{\overline\partial f}_{\mathrm{op}}+\norm{\overline\partial h}_{\mathrm{op}}\right)\right)\le\norm{\int \overline\partial(h-f)}_{C(\mathbb{C}\mathrm{P}_q^1)}=\norm{h - f - \psiinf(h-f)}_{C(\mathbb{C}\mathrm{P}_q^1)}.
\]
By our assumption,
\[
\norm{\overline\partial f}_{\mathrm{op}}+\norm{\overline\partial h}_{\mathrm{op}} 
< \frac{1 - \norm{h - f - \psiinf(h-f)}_{C(\mathbb{C}\mathrm{P}_q^1)}}{M},
\]
and hence
\[
1 - M\left(\norm{\overline\partial f}_{\mathrm{op}}+\norm{\overline\partial h}_{\mathrm{op}}\right)
> \norm{h - f - \psiinf(h-f)}_{C(\mathbb{C}\mathrm{P}_q^1)}.
\]
Therefore,
\begin{eqnarray}\label{gui}
\frac{\norm{h - f - \psiinf(h-f)}_{C(\mathbb{C}\mathrm{P}_q^1)}}
{1 - M\norm{\overline\partial f}_{\mathrm{op}}-M\norm{\overline\partial h}_{\mathrm{op}} } 
< 1,
\end{eqnarray}
which implies that $\norm{v_0}_{C(\mathbb{C}\mathrm{P}_q^1)} < 1$. Thus, $g_0 = v_0 - 1\in\mathscr{D}$ is an invertible element in the $C^*$-algebra $C(\mathbb{C}\mathrm{P}_q^1)$ and a solution of the mixed gauge equation $\overline\partial g = g \overline\partial f - \overline\partial h\,.\,g$.
\end{proof}

Notice that if $f-h\in\mathbb{C}$ for $f,h\in\mathcal{A}(\mathbb{C}\mathrm{P}_q^1)$, then the $\bar\partial$-connections $\overline{\nabla}^{(n)}_{\overline\partial f}\mbox{ and }\overline{\nabla}^{(n)}_{\overline\partial h}$ are automatically gauge equivalent, being the same.

\begin{proposition}\label{many gauge}
Suppose $f, h \in \mathcal{A}(\mathbb{C}\mathrm{P}_q^1)$ be such that $f-h$ is non-scalar and
\[
\norm{\overline\partial f}_{\mathrm{op}}+\norm{\overline\partial h}_{\mathrm{op}} 
< \min\left\{\frac{1}{3}\,,\,\frac{1 - 2\norm{h - f - \psiinf(h-f)}_{C(\mathbb{C}\mathrm{P}_q^1)}}{M}\right\}.
\]
Then, $\overline{\nabla}^{(n)}_{\overline\partial f}$ and $\overline{\nabla}^{(n)}_{\overline\partial h}$ are gauge equivalent holomorphic structures on $\mathcal{L}_n$.
\end{proposition}
\begin{proof}
In view of the previous lemma, it suffices to show that $\norm{v_0}_{C(\mathbb{C}\mathrm{P}_q^1)}<1/2$ and $\norm{\overline\partial v_0}_{\mathrm{op}}<1/2$. From \eqref{gui}, we immediately get $\norm{v_0}_{C(\mathbb{C}\mathrm{P}_q^1)}<1/2$ under the stronger assumption of the present. Since $v_0$ satisfies the equation in \eqref{lem:mixed-fixed-point}, we get
\begin{eqnarray*}
\norm{\overline\partial v_0}_{\mathrm{op}} &\leq& \norm{v_0}_{C(\mathbb{C}\mathrm{P}_q^1)}\left(\norm{\overline\partial f}_{\mathrm{op}}+\norm{\overline\partial h}_{\mathrm{op}}\right)+\norm{\overline\partial(h-f)}_{\mathrm{op}}\\
&\le& (\norm{v_0}_{C(\mathbb{C}\mathrm{P}_q^1)}+1)\left(\norm{\overline\partial f}_{\mathrm{op}}+\norm{\overline\partial h}_{\mathrm{op}}\right)\\
&<& \frac{1}{2}.
\end{eqnarray*}
So $\norm{v_0}_{\mathrm{Gr}}<1$, and we conclude that $g_0^{-1}=-(1-v_0)^{-1}\in\mathscr{D}$. Thus, we have an invertible element $g_0\in\mathscr{D}$ (namely, $g_0:=v_0-1$) that solves the mixed gauge equation $\overline\partial g=g\overline\partial f-\overline\partial h\,.\,g$, and hence $\overline{\nabla}^{(n)}_{\overline\partial f}$ and $\overline{\nabla}^{(n)}_{\overline\partial h}$ are gauge equivalent holomorphic structures by Definition \ref{main definition}.
\end{proof}

In the following, we indicate by example how Proposition \ref{many gauge} can be applied to produce many gauge equivalent $\bar\partial$-connections. These examples are beyond the scope of \cite{GravemanLaRueMacArthurPesinWei2025NonStandardHolomorphic}.

\begin{example}\label{example}
Fix any $m_0\in\mathbb{N}\setminus\{0\}$ and consider $x\in\mathcal{A}(\mathbb{C}\mathrm{P}_q^1)$ such that $\psiinf(x)=0$ (i.e. $x$ is non-scalar) and $\norm{x}_{C(\mathbb{C}\mathrm{P}_q^1)}<(2q^{m_0}(1-q))^{-1}$. Observe that $\norm{x}_{C(\mathbb{C}\mathrm{P}_q^1)}<(2q^{m}(1-q))^{-1}$ for all $m\ge m_0$ as $0<q<1$. Now, consider the sequences $\{f_m\},\,\{h_m\}\subseteq\mathcal{A}(\mathbb{C}\mathrm{P}_q^1)$ given by $$f_m:=\begin{cases}
q^mx\,,\quad m\ge m_0\\
x\,,\qquad 0\le m< m_0,
\end{cases}\quad \mbox{and}\quad h_m:=\begin{cases}
q^{m+1}x\,,\quad m\ge m_0\\
x\,,\quad\qquad 0\le m< m_0.
\end{cases}$$ Then $\norm{\overline\partial f_m}_{\mathrm{op}}+\norm{\overline\partial h_m}_{\mathrm{op}}=(q^m+q^{m+1})\norm{\overline\partial x}_{\mathrm{op}}$ for all $m\ge m_0$.
\smallskip

\noindent\textbf{Case~1:} If $M\ge 3$. In this case, for all $m\ge m_0$
\[
\frac{1 - 2\norm{h_m - f_m - \psiinf(h_m-f_m)}_{C(\mathbb{C}\mathrm{P}_q^1)}}{M}=\frac{1-2q^m(1-q)\norm{x}_{C(\mathbb{C}\mathrm{P}_q^1)}}{M}
\]
because $\psiinf(x)=0$. Then, for all $m\ge m_0$
\begin{eqnarray}\label{hh}
\hspace*{1cm}\min\left\{\frac{1}{3}\,,\,\frac{1 - 2\norm{h_m - f_m - \psiinf(h_m-f_m)}_{C(\mathbb{C}\mathrm{P}_q^1)}}{M}\right\}=\frac{1-2q^m(1-q)\norm{x}_{C(\mathbb{C}\mathrm{P}_q^1)}}{M}
\end{eqnarray}
as $M\ge 3$ and $0<q<1$. Now, observe that
\[
\frac{1-2q^m(1-q)\norm{x}_{C(\mathbb{C}\mathrm{P}_q^1)}}{q^m(1+q)}=\frac{1}{q^m(1+q)}-\frac{2(1-q)\norm{x}_{C(\mathbb{C}\mathrm{P}_q^1)}}{1+q}\to\infty
\]
as $\,m\to\infty$ because $0<q<1$. Therefore, there are infinitely many $m\in\mathbb{N}$ with $m\ge m_0$ for which $\frac{1-2q^m(1-q)\norm{x}_{C(\mathbb{C}\mathrm{P}_q^1)}}{q^m(1+q)}\gg M\norm{\bar\partial x}_{\mathrm{op}}$, since $M$ and $x$ are fixed. Choose one such $m^\prime\in\mathbb{N}$. This says that
\[
\norm{\overline\partial f_{m^\prime}}_{\mathrm{op}}+\norm{\overline\partial h_{m^\prime}}_{\mathrm{op}}=(q^{m^\prime}+q^{m^\prime+1})\norm{\bar\partial x}_{\mathrm{op}}<\frac{1-2q^{m^\prime}(1-q)\norm{x}_{C(\mathbb{C}\mathrm{P}_q^1)}}{M}\,.
\]
By \eqref{hh}, we see that the hypothesis of Proposition \ref{many gauge} is satisfied, and hence we have a pair $(f_{m^\prime},h_{m^\prime})$ such that $f_{m^\prime}-h_{m^\prime}$ is non-scalar (as $x$ is such), and $\overline{\nabla}^{(n)}_{\overline\partial f_{m^\prime}}\equiv_g\overline{\nabla}^{(n)}_{\overline\partial h_{m^\prime}}$. Also, the same conclusion holds for any $m>m^\prime$.
\smallskip

\noindent\textbf{Case~2:} If $M<3$. In this case, there exists $N\in\mathbb{N}$ such that for all $k\ge\max\{N,m_0\}$, we have
\begin{eqnarray}\label{hhhh}
0<\frac{2q^k(1-q)\norm{x}_{C(\mathbb{C}\mathrm{P}_q^1)}}{M}<\frac{1}{M}-\frac{1}{3}\,,
\end{eqnarray}
because $0<q<1$ and $M<3$. Using $\psiinf(x)=0$, we have from \eqref{hhhh}
\begin{eqnarray}\label{hhh}
& & \min\left\{\frac{1}{3}\,,\,\frac{1 - 2\norm{h_m - f_m - \psiinf(h_m-f_m)}_{C(\mathbb{C}\mathrm{P}_q^1)}}{M}\right\}\nonumber\\
&=& \min\left\{\frac{1}{3}\,,\,\frac{1-2q^m(1-q)\norm{x}_{C(\mathbb{C}\mathrm{P}_q^1)}}{M}\right\}\,=\,\frac{1}{3}
\end{eqnarray}
for all $m\ge\max\{N,m_0\}$. For all $m\ge\max\{N,m_0\}$, we already know that
\[
\norm{\overline\partial f_m}_{\mathrm{op}}+\norm{\overline\partial h_m}_{\mathrm{op}}=q^m(1+q)\norm{\bar\partial x}_{\mathrm{op}}\,.
\]
Since $q^m\to 0$ as $m\to\infty$, for large enough $m^\prime\in\mathbb{N}$ satisfying $m^\prime\ge\max\{N,m_0\}$, we obtain $\norm{\overline\partial f_{m^\prime}}_{\mathrm{op}}+\norm{\overline\partial h_{m^\prime}}_{\mathrm{op}}<\frac{1}{3}$. By \eqref{hhh}, we see that the hypothesis of Proposition \ref{many gauge} is satisfied. Then, $(f_{m^\prime},h_{m^\prime})$ is a pair of elements in $\mathcal{A}(\mathbb{C}\mathrm{P}_q^1)$ for which $f_{m^\prime}-h_{m^\prime}$ is non-scalar (as $x$ is such), and $\overline{\nabla}^{(n)}_{\overline\partial f_{m^\prime}}\equiv_g\overline{\nabla}^{(n)}_{\overline\partial h_{m^\prime}}$. Also, the same conclusion holds for any $m>m^\prime$.
\smallskip

In either case, we obtain pair $(f,h)$ of elements in $\mathcal{A}(\mathbb{C}\mathrm{P}_q^1)$, independent of the actual value of $M$, such that $f-h$ is non-scalar and $\overline{\nabla}^{(n)}_{\overline\partial f}\equiv_g\overline{\nabla}^{(n)}_{\overline\partial h}$.\qed
\end{example}

The above example indicates that the exact value of $M$ is not needed if one wishes to produce gauge equivalent $\bar\partial$-connections. Applying similar ideas as above, one can construct many other examples of gauge equivalent $\bar\partial$-connections. Another point is that the pair $(f,h)$ in the above example is constructed starting with an arbitrary element $x\in\mathcal{A}(\mathbb{C}\mathrm{P}_q^1)$, and not just from $\mathrm{Pol}(B_0)\subseteq\mathcal{A}(\mathbb{C}\mathrm{P}_q^1)$ (compare with Lemma \ref{complicacy}), and this is beyond the realm of \cite{GravemanLaRueMacArthurPesinWei2025NonStandardHolomorphic}.
\medskip


\section{A necessary and sufficient criterion for gauge equivalence}

In this section, we obtain a necessary and sufficient criterion in terms of fixed points of certain maps for the $\bar\partial$-connections $\overline{\nabla}_{\overline\partial f}\mbox{ and }\overline{\nabla}_{\overline\partial h}$ to be gauge equivalent, given $f,h\in\mathcal{A}(\mathbb{C}\mathrm{P}_q^1)$. Thus, the existence of invertible solutions of the mixed gauge equation is intimately related to the existence of fixed points of certain maps acting on an appropriate Banach space.

\subsection{A perturbative criterion}

Recall that $H=\mathscr{D}\cap\mathrm{ker}(\psi_\infty)$ is a Banach subspace of $\mathscr{D}$ equipped with the graph norm. For $f,h\in\mathcal{A}(\mathbb{C}\mathrm{P}_q^1)$ we have a map
\[ \Phi_{f,h}:H\longrightarrow H,\qquad
  \Phi_{f,h}(v) \;:=\;
  \int\bigl(v\,\overline\partial f \;-\; \overline\partial h\cdot v \;+\; \overline\partial(h-f)\bigr).
\]
The map is linear if and only if $h-f$ is a scalar. Note that $\mathrm{Ran}(\Phi_{f,h})\subseteq H$ follows from Proposition \ref{main 1} and Lemma \ref{main 2}.

Now, consider the linear operator
\[
  L_{f,h}:\mathscr{D}\longrightarrow \mathcal{B}(H_+,H_-),
  \qquad
  L_{f,h}(v) \;:=\; \overline\partial v \;-\; v\,\overline\partial f \;+\; \overline\partial h\cdot v,
\]
and denote $\beta_{f,h} \;:=\; \overline\partial(h-f) \;\in\;\bar\partial(\mathcal{A}(\mathbb{C}\mathrm{P}_q^1))\subseteq\mathcal{B}(H_+\,,\,H_-)$.

\begin{lemma}
The linear operator $L_{f,h}:\mathscr{D}\longrightarrow\mathcal{B}(H_+\,,\,H_-)$ is a bounded linear operator with respect to the graph norm on $\mathscr{D}$, and we have
\[
\norm{L_{f,h}}_{\mathrm{op}}\leq 1+\norm{\overline\partial f}_{\mathrm{op}}+\norm{\overline\partial h}_{\mathrm{op}}.
\]
\end{lemma}
\begin{proof}
This follows from the triangle inequality and the fact that $\norm{\overline\partial v}_{\mathrm{op}}\leq\norm{v}_{\mathrm{Gr}}$.
\end{proof}

Observe the following equivalent statements\,:
\begin{eqnarray*}
& & \Phi_{f,h}\mbox{ has a fixed point } v_0\in H\\
&\iff& \mbox{ the nonhomogeneous equation }L_{f,h}(v)=\beta_{f,h}\mbox{ has a solution } v_0\in H.\\
&\iff& \mbox{ the homogeneous equation }L_{f,h}(v)=0\mbox{ has a solution }\widetilde{v}_0:=v_0-1\in\mathscr{D}\\
& & \mbox{ with }\psiinf(\widetilde v_0)=-1.\\
&\iff& \mbox{ the mixed gauge equation has a solution }\widetilde{v}_0\in\mathscr{D}\mbox{ with }\psiinf(\widetilde v_0)=-1.
\end{eqnarray*}
Thus, for $f,h\in\mathcal{A}(\mathbb{C}\mathrm{P}_q^1)$, the map $\Phi_{f,h}:H\to H$ has a fixed point if and only if
\[
\beta_{f,h} \;\in\; \mathrm{Ran}\,(L_{f,h}|_H) \;\subseteq\,\mathcal{B}(H_+,H_-).
\]

\begin{lemma}
Given $f,h\in\mathcal{A}(\mathbb{C}\mathrm{P}_q^1)$, one has the following `iff' criteria\,:
\[
\begin{matrix}
g_0\in H\mbox{ solves the mixed gauge equation} &\iff& v_0\in\mathscr{D}\mbox{ solves }L_{f,h}(v)=\beta_{f,h}\\
 & & \mbox{ and }\psi_\infty(v_0)=1\\
  & & \\
g_0\in\mathscr{D}\setminus H\mbox{ solves the mixed gauge equation} &\iff& v_0\in H\mbox{ solves }L_{f,h}(v)=\beta_{f,h}\\
 & & \Big\Updownarrow\\
 & & \Phi_{f,h}\mbox{ has a fixed point } v_0\in H.
\end{matrix}
\]
\end{lemma}
\begin{proof}
The vertical `iff' statement is already observed previously. To prove the first horizontal `iff' statement, if $g_0\in H$ solves the mixed gauge equation, then observe that $v_0=g_0+1$ is the required solution of $L_{f,h}(v)=\beta_{f,h}$, and vice versa. 

To prove the second horizontal `iff' statement, if $g_0\in\mathscr{D}$ solves the mixed gauge equation, then $v_0=1+\frac{g_0}{\psi_\infty(g_0)}$ is a solution of $L_{f,h}(v)=\beta_{f,h}$ with $\psiinf(v_0)=2$, i.e., $v_0\notin H$. Now, observe that if $v_0\in\mathscr{D}$ solves $L_{f,h}(v)=\beta_{f,h}$ and $\psi_\infty(v_0)\neq 1$, then $v_1=\frac{v_0-\psi_\infty(v_0)}{1-\psi_\infty(v_0)}\in H$ and $L_{f,h}(v_1)=\beta_{f,h}$. For the converse direction, given such $v_0$, choose $g_0=v_0-1\in\mathscr{D}\setminus H$.
\end{proof}

Therefore, if $\overline{\nabla}_{\overline\partial f}\equiv_g\overline{\nabla}_{\overline\partial h}$, then $\Phi_{f,h}$ has a fixed point; or in other words, if $\Phi_{f,h}$ has no fixed point, then $\overline{\nabla}_{\overline\partial f}\not\equiv_g\overline{\nabla}_{\overline\partial h}$. However, note that this is not an `if and only if' criterion because to guarantee the gauge equivalence, we also need to assert the \emph{invertibility} condition of the solution to the mixed gauge equation. We shall obtain a necessary and sufficient criterion in the sequel.

The following result shows that fixed points of $\Phi_{f,h}$ cannot be arbitrarily small. Denote $\,\omega_{f,h}:=\int\beta_{f,h}=h-f-\psiinf(h-f)\in\mathcal{A}(\mathbb{C}\mathrm{P}_q^1)$.

\begin{lemma}\label{prop:norm-lower-bound}
For every $v\in H$ we have the estimate
\[
  \bigl\|\Phi_{f,h}(v)-v\bigr\|_{\mathrm{Gr}}
  \;\ge\;
  \|w_{f,h}\|_{\mathrm{Gr}}
  \;-\;
  (M+1)\,\|L_{f,h}\|_{\mathrm{op}}\,\|v\|_{\mathrm{Gr}}.
\]
In particular, if $v\in H$ is a fixed point of $\Phi_{f,h}$, then
\[
  \|v\|_{\mathrm{Gr}}
  \;\ge\;
  \frac{\norm{w_{f,h}}_{\mathrm{Gr}}}{(M+1)\norm{L_{f,h}}_{\mathrm{op}}}.
\]
\end{lemma}
\begin{proof}
Define the translation
\[
  \Phi_T:H\to H, \qquad \Phi_T(v):=v+w_{f,h}\,.
\]
We have for all $v\in H$,
\[
  \Phi_{f,h}(v)-\Phi_T(v)
  \;=\;
  -\,\int L_{f,h}(v).
\]
Equivalently,
\[
  \Phi_{f,h}(v)-v
  \;=\;
  -\,\int L_{f,h}(v) + w_{f,h}.
\]
Now,
\[
  \norm{\Phi_{f,h}(v)-v}_{\mathrm{Gr}}
  \;=\;
  \norm{-\,\int L_{f,h}(v) + w_{f,h}}_{\mathrm{Gr}}
  \;\ge\;
  \norm{w_{f,h}}_{\mathrm{Gr}} - \norm{\int(L_{f,h}(v))}_{\mathrm{Gr}}.
\]
By boundedness of $L_{f,h}$ and applying Lemma \ref{main 2}, we obtain
\[
  \norm{\int L_{f,h}(v)}_{\mathrm{Gr}}
  \;\le\;
  M\norm{L_{f,h}(v)}_{\mathrm{op}}+\norm{L_{f,h}(v)}_{\mathrm{op}}
  \;\le\;
  (M+1)\,\|L_{f,h}\|_{\mathrm{op}}\,\|v\|_{\mathrm{Gr}}.
\]
Substituting this into the previous inequality yields the result.
\end{proof}

\begin{corollary}
Given $f,h\in\mathcal{A}(\mathbb{C}\mathrm{P}_q^1)$ such that $h-f$ is non-scalar, set
\[
  \rho_{f,h}
  :=
  \frac{\|w_{f,h}\|_{\mathrm{Gr}}}{(M+1)\,\|L_{f,h}\|_{\mathrm{op}}}.
\]
Then, for every $0<\rho<\rho_{f,h}$ the map $\Phi_{f,h}$ has no fixed point in the closed ball
\[
  B_\rho
  :=
  \bigl\{\,v\in H \;:\; \|v\|_{\mathrm{Gr}}\le\rho\,\bigr\}.
\]
\end{corollary}


\subsection{A {\em necessary and sufficient} criterion for gauge equivalence}

Let $f,h\in\mathcal{A}(\mathbb{C}\mathrm{P}_q^1)$ be such that $h-f$ is non-scalar. Consider the following one-parameter family of maps:
\begin{align*}
\Phi_{f,h,\lambda}:H &\to H\\
v &\mapsto \int(v\overline\partial f-\overline\partial h\,.\,v+\lambda\,\overline\partial(h-f)),
\end{align*}
for $\lambda\in\mathbb{C}^\times:=\mathbb{C}\setminus\{0\}$. These maps are well-defined due to Proposition \ref{main 1} and Lemma \ref{main 2}. Note that our earlier focus in this section was confined to $\lambda=1$. Recall once again
\begin{align*}
L_{f,h}(v) &:= \overline\partial v-v\overline\partial f+\overline\partial h\,.\,v,\\
\beta_{f,h} &:= \overline\partial (h-f).
\end{align*}
It is easy to verify that $\Phi_{f,h,\lambda}$ has a fixed point in $H$ if and only if the nonhomogeneous equation $L_{f,h}(v)=\lambda\beta_{f,h}$ has a solution in $H$, using Lemma \ref{main 2}. Moreover, $L_{f,h}(v)=\lambda\beta_{f,h}$ if and only if $\lambda-v$ solves the homogeneous equation $L_{f,h}(x)=0$, which is precisely the mixed gauge equation.

Denote the open unit ball in $H$ by the following:
\[
B_1(0;H):=\{\xi\in H:\norm{\xi}_{\mathrm{Gr}}<1\}.
\]
Our final result is the following.

\begin{theorem}\label{main}
Given any $f,h\in\mathcal{A}(\mathbb{C}\mathrm{P}_q^1)$, the holomorphic structures $\overline{\nabla}^{(n)}_{\overline\partial f}$ and $\overline{\nabla}^{(n)}_{\overline\partial h}$ on $\mathcal{L}_n$ are gauge equivalent (i.e., $\overline{\nabla}^{(n)}_{\overline\partial f}\equiv_g\overline{\nabla}^{(n)}_{\overline\partial h}$) if and only if there exists $\lambda_0\in\mathbb{C}^\times$ such that the map $\Phi_{f,h,\lambda_0}$ has a fixed point $\xi\in B_1(0;H)$ with $\lambda_0-\xi$ invertible in $\mathscr{D}$. Moreover, such a $\lambda_0$ can be chosen to satisfy $|\lambda_0|<1$.
\end{theorem}
\begin{proof}
The case of $f-h$ being a scalar is trivial. Therefore, assume that $f-h$ is non-scalar in $\mathcal{A}(\mathbb{C}\mathrm{P}_q^1)$.

First, assume that $\overline{\nabla}^{(n)}_{\overline\partial f}\equiv_g\overline{\nabla}^{(n)}_{\overline\partial h}$. Then we have an invertible solution $v\in\mathscr{D}$ for the homogeneous equation $L_{f,h}(v)=0$. Moreover, $\psiinf(v)\neq 0$ as $v$ cannot lie in $H$. Observe that $v$ is not a scalar as $f-h$ is non-scalar, and hence $v\neq\psiinf(v)$. In the following, we simply write $\norm{\cdot}$ in place of $\norm{\cdot}_{\mathrm{Gr}}$, as no confusion arises.
\smallskip

\noindent\textit{Claim:} There exist nonzero scalars $\alpha,\gamma$ such that $v^\prime=\alpha v+\gamma$ is also invertible and $\norm{v^\prime-\psiinf(v^\prime)}<|\psiinf(v^\prime)|$.
\smallskip

Observe that $v^\prime-\psiinf(v^\prime)=\alpha(v-\psiinf(v))$. So the requirement $\norm{v^\prime-\psiinf(v^\prime)}<|\psiinf(v^\prime)|$ is satisfied if and only if $|\alpha|\norm{v-\psiinf(v)}<|\alpha\psiinf(v)+\gamma|$. Also, $v^\prime=\alpha(v+\gamma/\alpha)$ is invertible if and only if $-\gamma/\alpha\notin\sigma(v)$. Choose any nonzero scalar $\alpha\in\mathbb{C}$ and denote $C:=|\alpha|\norm{v-\psiinf(v)}$. Since $v\neq\psiinf(v)$, we have $C>0$. Since the map $z\mapsto \alpha\psiinf(v)+z$ is bijective from $\mathbb{C}$ to $\mathbb{C}$, choose a nonzero scalar $\gamma$ such that $|\alpha\psiinf(v)+\gamma|>C=|\alpha|\norm{v-\psiinf(v)}$. Observe that for fixed $z_0=\alpha\psiinf(v)$, the set $\{z\in\mathbb{C}:|z+z_0|>C\}$ is unbounded since it contains the complement of the closed  ball $\{z\in\mathbb{C}:|z|\le C+|z_0|\}$. Since $\sigma(v)\subseteq\mathbb{C}$ is a bounded set, there exists $\gamma\in\mathbb{C}\setminus\{0\}$ such that the conditions $-\gamma/\alpha\notin\sigma(v)$ and $|\alpha\psiinf(v)+\gamma|>C$ simultaneously hold. This proves that $v^\prime$ is invertible in $\mathscr{D}$ and $\norm{v^\prime-\psiinf(v^\prime)}=C<|\alpha\psiinf(v)+\gamma|=|\psiinf(v^\prime)|$, which completes the proof of the claim.

Notice that $v^\prime$ is invertible and hence $\psiinf(v^\prime)\neq 0$. Consider $\xi=1-\frac{v^\prime}{\psiinf(v^\prime)}$. Then, clearly $\xi\in H$ as $\psiinf$ is linear, and by the above claim we have $\norm{\xi}<1$. Thus, $\xi$ lies in the open unit ball in $H$. Now,
\begin{eqnarray*}
L_{f,h}(\xi) &=& \overline\partial \xi-\xi\overline\partial f+\overline\partial h\,.\,\xi\\
&=& -\frac{1}{\psiinf(v^\prime)}(\overline\partial v^\prime-v^\prime\overline\partial f+\overline\partial h\,.\,v^\prime)+\beta_{f,h}.
\end{eqnarray*}
Since $v$ satisfies $L_{f,h}(v)=0$, we have $L_{f,h}(v^\prime)=\gamma\beta_{f,h}$. Then
\[
L_{f,h}(\xi)=\left(1-\frac{\gamma}{\psiinf(v^\prime)}\right)\beta_{f,h}.
\]
Observe that $\gamma\neq\psiinf(v^\prime)$, since $\alpha\neq 0$ and $\psiinf(v)\neq 0$. Therefore, $\xi$ is a fixed point of $\Phi_{f,h,\lambda_0}$ in the open unit ball of $H$, where $\lambda_0=1-\frac{\gamma}{\psiinf(v^\prime)}\in\mathbb{C}\setminus\{0\}$. Observe that $\gamma\notin\sigma(v^\prime)=\alpha\sigma(v)+\gamma$. This is because $v$ was invertible and therefore $0\notin\sigma(v)$. Therefore, $\lambda_0\notin\sigma(\xi)$, i.e., $\lambda_0-\xi$ is invertible.

Moreover, we can also choose $\gamma$ in such a way that $|\lambda_0|<1$. For this, write $\alpha\psiinf(v)=z_0$ fixed. We have $-\gamma/\alpha\notin\sigma(v)$ and $|z_0+\gamma|>C=|\alpha|\norm{v-\psiinf(v)}$. Observe that $|\lambda_0|<1$ if and only if $|z_0+\gamma|>|z_0|$. Let $M=\max\{|z_0|,C\}$ and choose $\gamma^{\,\prime}$ that satisfies $|\gamma^{\,\prime}|>|\alpha|R+M+|z_0|$, where $R>0$ is such that $\sigma(v)\subseteq B_R(0)$. This is always possible, as $\sigma(v)$ is a bounded subset of $\mathbb{C}$. Then, $-\gamma^{\,\prime}/\alpha\notin\sigma(v)$ as $|\gamma^{\,\prime}|>|\alpha|R$, and
\begin{eqnarray}\label{final inequality}
|z_0+\gamma^{\,\prime}|\ge|\gamma^{\,\prime}|-|z_0|>|\alpha|R+M\ge\max\{|z_0|,C\}.
\end{eqnarray}
Therefore, if we choose $v^{\prime\prime}=\alpha v+\gamma^{\,\prime}$, then $v^{\prime\prime}$ is invertible as $-\gamma^{\,\prime}/\alpha\notin\sigma(v)$ and $\norm{v^{\prime\prime}-\psiinf(v^{\prime\prime})}<|\psiinf(v^{\prime\prime})|$ using \eqref{final inequality}. Thus, all the relevant properties satisfied by $v^\prime$ are also satisfied by $v^{\prime\prime}$, and we additionally have $|\lambda_0|<1$.
\medskip

For the converse direction, suppose that there exists $\lambda_0$ such that $\Phi_{f,h,\lambda_0}$ has a fixed point $\xi\in B_1(0;H)$ with $\lambda_0-\xi$ invertible in $\mathscr{D}$. Then $L_{f,h}(\xi)=\lambda_0\beta_{f,h}$. Using Lemma \ref{main 2}, it follows that $x:=\lambda_0-\xi$ is an invertible solution of the mixed gauge equation
\[
\overline\partial x=x\overline\partial f-\overline\partial h\,.\,x\,,
\]
and consequently, $\overline{\nabla}^{(n)}_{\overline\partial f}\equiv_g\overline{\nabla}^{(n)}_{\overline\partial h}$.
\end{proof}

\bigskip

\bibliography{biblio.bib}

@article{KhalkhaliLandiVanSuijlekom2011,
    AUTHOR = {Khalkhali, M. and Landi, G. and van Suijlekom,
              W.},
     TITLE = {Holomorphic structures on the quantum projective line},
   JOURNAL = {Int. Math. Res. Not. IMRN},
  FJOURNAL = {International Mathematics Research Notices. IMRN},
      YEAR = {2011},
    NUMBER = {4},
     PAGES = {851--884},
      ISSN = {1073-7928,1687-0247},
   MRCLASS = {58B34 (16E40 58B32)},
  MRNUMBER = {2773332},
MRREVIEWER = {Snigdhayan\ Mahanta},
       DOI = {10.1093/imrn/rnq097},
       URL = {https://doi.org/10.1093/imrn/rnq097},
}

@article{GravemanLaRueMacArthurPesinWei2025NonStandardHolomorphic,
  title         = {Non-standard {H}olomorphic {S}tructures on {L}ine {B}undles over the {Q}uantum {P}rojective {L}ine},
  author        = {Graveman, M. and La Rue, L. and MacArthur, L. and Pesin, H. and Wei, Z.},
  year          = {2025},
  eprint        = {2510.14263},
  archivePrefix = {arXiv},
  primaryClass  = {math.QA},
  note          = {arXiv:2510.14263v2},
}

@article{AguilarKaad2018PodlesSpectralMetric,
    AUTHOR = {Aguilar, K. and Kaad, J.},
     TITLE = {The {P}odle\'s{} sphere as a spectral metric space},
   JOURNAL = {J. Geom. Phys.},
  FJOURNAL = {Journal of Geometry and Physics},
    VOLUME = {133},
      YEAR = {2018},
     PAGES = {260--278},
      ISSN = {0393-0440,1879-1662},
   MRCLASS = {58B34 (46L30 46L89 58B32)},
  MRNUMBER = {3850270},
MRREVIEWER = {Ferdinand\ Ngakeu},
       DOI = {10.1016/j.geomphys.2018.07.015},
       URL = {https://doi.org/10.1016/j.geomphys.2018.07.015},
}

@book {Connes,
    AUTHOR = {Connes, A.},
     TITLE = {Noncommutative geometry},
 PUBLISHER = {Academic Press, Inc., San Diego, CA},
      YEAR = {1994},
     PAGES = {xiv+661},
      ISBN = {0-12-185860-X},
   MRCLASS = {46Lxx (19K56 22D25 58B30 58G12 81T13 81V22 81V70)},
  MRNUMBER = {1303779},
MRREVIEWER = {John\ Roe},
}

@incollection {DAndreaLandi2013GeometryQuantumProjectiveSpaces,
    AUTHOR = {D'Andrea, F. and Landi, G.},
     TITLE = {Geometry of quantum projective spaces},
 BOOKTITLE = {Noncommutative geometry and physics. 3},
    SERIES = {Keio COE Lect. Ser. Math. Sci.},
    VOLUME = {1},
     PAGES = {373--416},
 PUBLISHER = {World Sci. Publ., Hackensack, NJ},
      YEAR = {2013},
      ISBN = {978-981-4425-00-1},
   MRCLASS = {58B34 (19K99 58J50 81T75)},
  MRNUMBER = {3098599},
       DOI = {10.1142/9789814425018\_0014},
       URL = {https://doi.org/10.1142/9789814425018_0014},
}

@article {DAndreya-Dabrowski,
    AUTHOR = {D'Andrea, F. and D\polhk{a}browski, L.},
     TITLE = {Dirac operators on quantum projective spaces},
   JOURNAL = {Comm. Math. Phys.},
  FJOURNAL = {Communications in Mathematical Physics},
    VOLUME = {295},
      YEAR = {2010},
    NUMBER = {3},
     PAGES = {731--790},
      ISSN = {0010-3616,1432-0916},
   MRCLASS = {58J42 (58B34)},
  MRNUMBER = {2600033},
MRREVIEWER = {Chiara\ Pagani},
       DOI = {10.1007/s00220-010-0989-8},
       URL = {https://doi.org/10.1007/s00220-010-0989-8},
}

@article{KhalkhaliMoatadelro2011NCComplexGeometryQPS,
    AUTHOR = {Khalkhali, M. and Moatadelro, A.},
     TITLE = {Noncommutative complex geometry of the quantum projective
              space},
   JOURNAL = {J. Geom. Phys.},
  FJOURNAL = {Journal of Geometry and Physics},
    VOLUME = {61},
      YEAR = {2011},
    NUMBER = {12},
     PAGES = {2436--2452},
      ISSN = {0393-0440,1879-1662},
   MRCLASS = {58B34 (16E40 16T20 58B32)},
  MRNUMBER = {2838520},
MRREVIEWER = {Chiara\ Pagani},
       DOI = {10.1016/j.geomphys.2011.08.004},
       URL = {https://doi.org/10.1016/j.geomphys.2011.08.004},
}

@article {Mesland-Rennie,
    AUTHOR = {Mesland, B. and Rennie, A.},
     TITLE = {Curvature and {W}eitzenb\"ock formula for the {P}odle\'s{}
              quantum sphere},
   JOURNAL = {Comm. Math. Phys.},
  FJOURNAL = {Communications in Mathematical Physics},
    VOLUME = {406},
      YEAR = {2025},
    NUMBER = {8},
     PAGES = {Paper No. 179, 36},
      ISSN = {0010-3616,1432-0916},
   MRCLASS = {58B34 (46L87)},
  MRNUMBER = {4927808},
       DOI = {10.1007/s00220-025-05348-9},
       URL = {https://doi.org/10.1007/s00220-025-05348-9},
}

@article{Podles1987QuantumSpheres,
    AUTHOR = {Podle\'s, P.},
     TITLE = {Quantum spheres},
   JOURNAL = {Lett. Math. Phys.},
  FJOURNAL = {Letters in Mathematical Physics. A Journal for the Rapid
              Dissemination of Short Contributions in the Field of
              Mathematical Physics},
    VOLUME = {14},
      YEAR = {1987},
    NUMBER = {3},
     PAGES = {193--202},
      ISSN = {0377-9017},
   MRCLASS = {46L55 (22D25 22E65 81D99)},
  MRNUMBER = {919322},
MRREVIEWER = {Ola\ Bratteli},
       DOI = {10.1007/BF00416848},
       URL = {https://doi.org/10.1007/BF00416848},
}

@article {Polishchuk2006AnaloguesExponentialMap,
    AUTHOR = {Polishchuk, A.},
     TITLE = {Analogues of the exponential map associated with complex
              structures on noncommutative two-tori},
   JOURNAL = {Pacific J. Math.},
  FJOURNAL = {Pacific Journal of Mathematics},
    VOLUME = {226},
      YEAR = {2006},
    NUMBER = {1},
     PAGES = {153--178},
      ISSN = {0030-8730,1945-5844},
   MRCLASS = {58B34 (46L87)},
  MRNUMBER = {2247860},
MRREVIEWER = {Matilde\ Marcolli},
       DOI = {10.2140/pjm.2006.226.153},
       URL = {https://doi.org/10.2140/pjm.2006.226.153},
}

@article {MR1695097,
    AUTHOR = {Fr\"ohlich, J. and Grandjean, O. and Recknagel, A.},
     TITLE = {Supersymmetric quantum theory and non-commutative geometry},
   JOURNAL = {Comm. Math. Phys.},
  FJOURNAL = {Communications in Mathematical Physics},
    VOLUME = {203},
      YEAR = {1999},
    NUMBER = {1},
     PAGES = {119--184},
      ISSN = {0010-3616,1432-0916},
   MRCLASS = {58J42 (46L87 46N50 58B34 81R60 81T60 81T75)},
  MRNUMBER = {1695097},
MRREVIEWER = {Alan\ L.\ Carey},
       DOI = {10.1007/s002200050608},
       URL = {https://doi.org/10.1007/s002200050608},
}

@article {MR683171,
    AUTHOR = {Witten, E.},
     TITLE = {Supersymmetry and {M}orse theory},
   JOURNAL = {J. Differential Geom.},
  FJOURNAL = {Journal of Differential Geometry},
    VOLUME = {17},
      YEAR = {1982},
    NUMBER = {4},
     PAGES = {661--692},
      ISSN = {0022-040X,1945-743X},
   MRCLASS = {58G99 (53C99 81G20)},
  MRNUMBER = {683171},
MRREVIEWER = {I.\ Vaisman},
       URL = {http://projecteuclid.org/euclid.jdg/1214437492},
}

@article {Polishchuk-Schwarz,
    AUTHOR = {Polishchuk, A. and Schwarz, A.},
     TITLE = {Categories of holomorphic vector bundles on noncommutative
              two-tori},
   JOURNAL = {Comm. Math. Phys.},
  FJOURNAL = {Communications in Mathematical Physics},
    VOLUME = {236},
      YEAR = {2003},
    NUMBER = {1},
     PAGES = {135--159},
      ISSN = {0010-3616,1432-0916},
   MRCLASS = {58B34 (18E30 32L05)},
  MRNUMBER = {1977884},
MRREVIEWER = {Matilde\ Marcolli},
       DOI = {10.1007/s00220-003-0813-9},
       URL = {https://doi.org/10.1007/s00220-003-0813-9},
}

@article {Polishchuk,
    AUTHOR = {Polishchuk, A.},
     TITLE = {Classification of holomorphic vector bundles on noncommutative
              two-tori},
   JOURNAL = {Doc. Math.},
  FJOURNAL = {Documenta Mathematica},
    VOLUME = {9},
      YEAR = {2004},
     PAGES = {163--181},
      ISSN = {1431-0635,1431-0643},
   MRCLASS = {58B34 (16E30 46L87 46M18)},
  MRNUMBER = {2054986},
MRREVIEWER = {Mauro\ Spera},
}

@article {Buachalla,
    AUTHOR = {\'O{} Buachalla, R.},
     TITLE = {Noncommutative {K}\"ahler structures on quantum homogeneous
              spaces},
   JOURNAL = {Adv. Math.},
  FJOURNAL = {Advances in Mathematics},
    VOLUME = {322},
      YEAR = {2017},
     PAGES = {892--939},
      ISSN = {0001-8708,1090-2082},
   MRCLASS = {58B34 (16T05 81R50 81R60)},
  MRNUMBER = {3720811},
MRREVIEWER = {Chiara\ Pagani},
       DOI = {10.1016/j.aim.2017.09.031},
       URL = {https://doi.org/10.1016/j.aim.2017.09.031},
}

@article {satyajit,
    AUTHOR = {Guin, S.},
     TITLE = {Noncommutative {K}\"ahler structure on {$C^*$}-dynamical
              systems},
   JOURNAL = {J. Geom. Phys.},
  FJOURNAL = {Journal of Geometry and Physics},
    VOLUME = {146},
      YEAR = {2019},
     PAGES = {103492, 24},
      ISSN = {0393-0440,1879-1662},
   MRCLASS = {58B34 (37A55 46L87)},
  MRNUMBER = {3995285},
       DOI = {10.1016/j.geomphys.2019.103492},
       URL = {https://doi.org/10.1016/j.geomphys.2019.103492},
}

@article {Schmudgen-Wagner,
    AUTHOR = {Schm\"udgen, K. and Wagner, E.},
     TITLE = {Dirac operator and a twisted cyclic cocycle on the standard
              {P}odle\'s{} quantum sphere},
   JOURNAL = {J. Reine Angew. Math.},
  FJOURNAL = {Journal f\"ur die Reine und Angewandte Mathematik. [Crelle's
              Journal]},
    VOLUME = {574},
      YEAR = {2004},
     PAGES = {219--235},
      ISSN = {0075-4102,1435-5345},
   MRCLASS = {58B34 (58B32 58J42)},
  MRNUMBER = {2099116},
MRREVIEWER = {Partha\ Sarathi\ Chakraborty},
       DOI = {10.1515/crll.2004.072},
       URL = {https://doi.org/10.1515/crll.2004.072},
}

@article {Neshveyev,
    AUTHOR = {Neshveyev, S. and Tuset, L.},
     TITLE = {A local index formula for the quantum sphere},
   JOURNAL = {Comm. Math. Phys.},
  FJOURNAL = {Communications in Mathematical Physics},
    VOLUME = {254},
      YEAR = {2005},
    NUMBER = {2},
     PAGES = {323--341},
      ISSN = {0010-3616,1432-0916},
   MRCLASS = {58J42 (46L87 58B34 58J22)},
  MRNUMBER = {2117628},
MRREVIEWER = {Yuri\ A.\ Kordyukov},
       DOI = {10.1007/s00220-004-1154-z},
       URL = {https://doi.org/10.1007/s00220-004-1154-z},
}

@incollection {Dabrowski-Sitarz,
    AUTHOR = {D\polhk{a}browski, L. and Sitarz, A.},
     TITLE = {Dirac operator on the standard {P}odle\'s{} quantum sphere},
 BOOKTITLE = {Noncommutative geometry and quantum groups ({W}arsaw, 2001)},
    SERIES = {Banach Center Publ.},
    VOLUME = {61},
     PAGES = {49--58},
 PUBLISHER = {Polish Acad. Sci. Inst. Math., Warsaw},
      YEAR = {2003},
   MRCLASS = {58J42 (58B34 81R60)},
  MRNUMBER = {2024421},
MRREVIEWER = {Alexander\ Gorokhovsky},
       DOI = {10.4064/bc61-0-4},
       URL = {https://doi.org/10.4064/bc61-0-4},
}

@article {Rieffel,
    AUTHOR = {Rieffel, M. A.},
     TITLE = {Gromov-{H}ausdorff distance for quantum metric spaces},
   JOURNAL = {Mem. Amer. Math. Soc.},
  FJOURNAL = {Memoirs of the American Mathematical Society},
    VOLUME = {168},
      YEAR = {2004},
    NUMBER = {796},
     PAGES = {1--65},
      ISSN = {0065-9266,1947-6221},
   MRCLASS = {46L87 (53C23 58B34 60B10)},
  MRNUMBER = {2055927},
       DOI = {10.1090/memo/0796},
       URL = {https://doi.org/10.1090/memo/0796},
}

@book {Klimyk,
    AUTHOR = {Klimyk, A. and Schm\"udgen, K.},
     TITLE = {Quantum groups and their representations},
    SERIES = {Texts and Monographs in Physics},
 PUBLISHER = {Springer-Verlag, Berlin},
      YEAR = {1997},
     PAGES = {xx+552},
      ISBN = {3-540-63452-5},
   MRCLASS = {17B37 (16W30 17B81 46L87 58B30 81R50)},
  MRNUMBER = {1492989},
MRREVIEWER = {Ian\ M.\ Musson},
       DOI = {10.1007/978-3-642-60896-4},
       URL = {https://doi.org/10.1007/978-3-642-60896-4},
}

@article {Woronowicz,
    AUTHOR = {Woronowicz, S. L.},
     TITLE = {Twisted {${\rm SU}(2)$} group. {A}n example of a
              noncommutative differential calculus},
   JOURNAL = {Publ. Res. Inst. Math. Sci.},
  FJOURNAL = {Kyoto University. Research Institute for Mathematical
              Sciences. Publications},
    VOLUME = {23},
      YEAR = {1987},
    NUMBER = {1},
     PAGES = {117--181},
      ISSN = {0034-5318,1663-4926},
   MRCLASS = {46L99 (22D25 22E45 46L80 58H05)},
  MRNUMBER = {890482},
MRREVIEWER = {Marc\ Rosso},
       DOI = {10.2977/prims/1195176848},
       URL = {https://doi.org/10.2977/prims/1195176848},
}

@article {Connes-Cuntz,
    AUTHOR = {Connes, A. and Cuntz, J.},
     TITLE = {Quasi homomorphismes, cohomologie cyclique et positivit\'e},
   JOURNAL = {Comm. Math. Phys.},
  FJOURNAL = {Communications in Mathematical Physics},
    VOLUME = {114},
      YEAR = {1988},
    NUMBER = {3},
     PAGES = {515--526},
      ISSN = {0010-3616,1432-0916},
   MRCLASS = {46L80 (19K99 58B15 58G12)},
  MRNUMBER = {929142},
MRREVIEWER = {Jonathan\ M.\ Rosenberg},
       URL = {http://projecteuclid.org/euclid.cmp/1104160692},
}

@article {Majid,
    AUTHOR = {Majid, S.},
     TITLE = {Noncommutative {R}iemannian and spin geometry of the standard
              {$q$}-sphere},
   JOURNAL = {Comm. Math. Phys.},
  FJOURNAL = {Communications in Mathematical Physics},
    VOLUME = {256},
      YEAR = {2005},
    NUMBER = {2},
     PAGES = {255--285},
      ISSN = {0010-3616,1432-0916},
   MRCLASS = {58B32 (58B34 81R50)},
  MRNUMBER = {2160795},
MRREVIEWER = {Ferdinand\ Ngakeu},
       DOI = {10.1007/s00220-005-1295-8},
       URL = {https://doi.org/10.1007/s00220-005-1295-8},
}

@article {Podles2,
    AUTHOR = {Podle\'s, P.},
     TITLE = {Differential calculus on quantum spheres},
   JOURNAL = {Lett. Math. Phys.},
  FJOURNAL = {Letters in Mathematical Physics},
    VOLUME = {18},
      YEAR = {1989},
    NUMBER = {2},
     PAGES = {107--119},
      ISSN = {0377-9017,1573-0530},
   MRCLASS = {58B30 (17B37 46L89 58G25)},
  MRNUMBER = {1010990},
       DOI = {10.1007/BF00401865},
       URL = {https://doi.org/10.1007/BF00401865},
}

@article {BM1,
    AUTHOR = {Brzezi\'nski, T. and Majid, S.},
     TITLE = {Quantum group gauge theory on quantum spaces},
   JOURNAL = {Comm. Math. Phys.},
  FJOURNAL = {Communications in Mathematical Physics},
    VOLUME = {157},
      YEAR = {1993},
    NUMBER = {3},
     PAGES = {591--638},
      ISSN = {0010-3616,1432-0916},
   MRCLASS = {58B30 (16W30 17B37 81R50)},
  MRNUMBER = {1243712},
MRREVIEWER = {Pavel\ \v S\v tov\'i\v cek},
       URL = {http://projecteuclid.org/euclid.cmp/1104254023},
}

@article {BM2,
    AUTHOR = {Brzezi\'nski, T. and Majid, S.},
     TITLE = {Quantum differentials and the {$q$}-monopole revisited},
   JOURNAL = {Acta Appl. Math.},
  FJOURNAL = {Acta Applicandae Mathematicae},
    VOLUME = {54},
      YEAR = {1998},
    NUMBER = {2},
     PAGES = {185--232},
      ISSN = {0167-8019,1572-9036},
   MRCLASS = {58B32 (20G42 81R50)},
  MRNUMBER = {1660923},
MRREVIEWER = {Albert\ Jeu-Liang\ Sheu},
       DOI = {10.1023/A:1006053806824},
       URL = {https://doi.org/10.1023/A:1006053806824},
}

@book {Sakai,
    AUTHOR = {Sakai, Sh\^oichir\^o},
     TITLE = {Operator algebras in dynamical systems},
    SERIES = {Encyclopedia of Mathematics and its Applications},
    VOLUME = {41},
      NOTE = {The theory of unbounded derivations in $C^*$-algebras},
 PUBLISHER = {Cambridge University Press, Cambridge},
      YEAR = {1991},
     PAGES = {xii+219},
      ISBN = {0-521-40096-1},
   MRCLASS = {46L57 (46-02 46L60 47D45 82B10 82B26)},
  MRNUMBER = {1136257},
MRREVIEWER = {Palle\ E. T. Jorgensen},
       DOI = {10.1017/CBO9780511662218},
       URL = {https://doi.org/10.1017/CBO9780511662218},
}

@article {Landi,
    AUTHOR = {Landi, G.},
     TITLE = {Equivariant and holomorphic bundles on the quantum projective
              line},
   JOURNAL = {Int. J. Geom. Methods Mod. Phys.},
  FJOURNAL = {International Journal of Geometric Methods in Modern Physics},
    VOLUME = {9},
      YEAR = {2012},
    NUMBER = {2},
     PAGES = {1260001, 8},
      ISSN = {0219-8878,1793-6977},
   MRCLASS = {58B34 (46L52 58B12)},
  MRNUMBER = {2913139},
MRREVIEWER = {Alan\ L.\ Carey},
       DOI = {10.1142/S0219887812600018},
       URL = {https://doi.org/10.1142/S0219887812600018},
}

@article {Kaad-Kyed,
    AUTHOR = {Aguilar, K. and Kaad, J. and Kyed, D.},
     TITLE = {The {P}odle\'s{} spheres converge to the sphere},
   JOURNAL = {Comm. Math. Phys.},
  FJOURNAL = {Communications in Mathematical Physics},
    VOLUME = {392},
      YEAR = {2022},
    NUMBER = {3},
     PAGES = {1029--1061},
      ISSN = {0010-3616,1432-0916},
   MRCLASS = {46L67 (58B32)},
  MRNUMBER = {4426737},
       DOI = {10.1007/s00220-022-04363-4},
       URL = {https://doi.org/10.1007/s00220-022-04363-4},
}

@article {Beggs-Smith,
    AUTHOR = {Beggs, E. and Paul Smith, S.},
     TITLE = {Non-commutative complex differential geometry},
   JOURNAL = {J. Geom. Phys.},
  FJOURNAL = {Journal of Geometry and Physics},
    VOLUME = {72},
      YEAR = {2013},
     PAGES = {7--33},
      ISSN = {0393-0440,1879-1662},
   MRCLASS = {17D99 (14A22 32C35 32Q60 53C56 58B34)},
  MRNUMBER = {3073899},
MRREVIEWER = {Andrey\ Yu.\ Lazarev},
       DOI = {10.1016/j.geomphys.2013.03.018},
       URL = {https://doi.org/10.1016/j.geomphys.2013.03.018},
}

\end{document}